\documentclass[a4paper,10pt]{article}

\usepackage{geometry}
\usepackage{epsfig}
\usepackage{amsmath}
\usepackage{amssymb}
\usepackage{amsthm}
\usepackage{mathrsfs}
\usepackage{color}
\usepackage{amscd}

\usepackage{amsthm}
\newtheorem{theorem}{Theorem}[section]
\newtheorem{definition}[theorem]{Definition}
\newtheorem{lemma}[theorem]{Lemma}
\newtheorem{proposition}[theorem]{Proposition}
\newtheorem{corollary}[theorem]{Corollary}
\newtheorem{remark}[theorem]{Remark}
\newtheorem{example}[theorem]{Example}
\newtheorem{examples}[theorem]{Examples}

\usepackage{amsfonts}

\newcommand{\oo}{{\mathbb{O}}}
\newcommand{\hh}{{\mathbb{H}}}

\newcommand{\cc}{{\mathbb{C}}}
\newcommand{\rr}{{\mathbb{R}}}

\newcommand{\nn}{{\mathbb{N}}}
\newcommand{\D}{\mathbb{D}}

\newcommand{\s}{{\mathbb{S}}}

\newcommand{\I}{\mathcal{I}}
\newcommand{\study}{{\mathscr{S}}}

\newcommand\vs[1]{{#1}_s^\circ}

\newcommand{\punto}{\bullet}

\newcommand\re{\operatorname{Re}}
\newcommand\im{\operatorname{Im}}
\newcommand\lra{\longrightarrow}
\newcommand{\ui}{\imath}

\newcommand{\OO}{\Omega}

\newcommand{\mbb}{\mathbb}

\newcommand{\mr}{\mathrm}
\newcommand{\mscr}{\mathscr}
\newcommand{\R}{\mbb{R}}
\newcommand{\mc}{\mathcal}
\newcommand{\hslashslash}{%
  \raisebox{.9ex}{%
    \scalebox{.7}{%
      \rotatebox[origin=c]{18}{$-$}%
    }%
  }%
}
\newcommand{\fslash}{%
  {%
   \vphantom{f}%
   \ooalign{\kern.05em\smash{\hslashslash}\hidewidth\cr$f$\cr}%
   \kern.05em
  }%
}

\title{\bf The algebra of slice functions}

\author{Riccardo Ghiloni\\
 Alessandro Perotti\\
\small Dipartimento di Matematica, Universit\`a di Trento\\ 
\small Via Sommarive 14, I-38123 Povo Trento, Italy\\
\small riccardo.ghiloni@unitn.it, alessandro.perotti@unitn.it\\
\and
Caterina Stoppato
\\ 
\small Istituto Nazionale di Alta Matematica\\
\small Unit\`a di Ricerca di Firenze c/o \\
\small DiMaI ``U. Dini'' Universit\`a di Firenze \\
\small Viale Morgagni 67/A, I-50134 Firenze, Italy\\
\small stoppato@math.unifi.it}

\date{  }

\begin{document}

\maketitle


\begin{abstract}
In this paper we study some fundamental algebraic properties of slice functions and slice regular functions over an alternative $^*$-algebra $A$ over $\rr$. These recently introduced function theories generalize to higher dimensions the classical theory of functions of a complex variable. Slice functions over $A$, which comprise all polynomials over $A$, form an alternative $^*$-algebra themselves when endowed with appropriate operations. We presently study this algebraic structure in detail and we confront with questions about the existence of multiplicative inverses. This study leads us to a detailed investigation of the zero sets of slice functions and of slice regular functions, which are of course of independent interest.
\end{abstract}


\section*{Introduction}

In the theory of functions of one complex variable, the notion of holomorphy plays a leading role not only for its analytic meaning but also because of its algebraic repercussions. Indeed, the discreteness of the zero sets of holomorphic functions allows to consider quotients of functions, which are an important tool in the construction of the theory. On the contrary, without the holomorphy assumption it is easy to choose nonzero elements $f,g$ of the algebra of functions $\cc \to \cc$  that are not identically zero but have $fg\equiv0$; that is, elements $f,g\not \equiv 0$ of the algebra that are zero divisors and therefore do not admit multiplicative inverses. As the theory of holomorphy built on, the fertile interplay between its analytic and algebraic sides remained one of its distinctive traits.

In dimensions higher than two -- chiefly over the algebras of quaternions $\hh$, that of octonions $\oo$ and over the Clifford algebras $C\ell_{0,m}=\R_m$ -- several theories introduced during the last century have met with success in matching many analytic features of the theory of holomorphic functions of one complex variable. Among the monographs on this subject, let us mention \cite{librosommen,GHS}. On the other hand, the search for an approach better adapted to certain algebraic requirements led to introducing and developing, over the last decade, a new function theory.

One of such requirements was, to include the classical theory of polynomials. As is well known (cf.~\cite[\S 16]{Lam}), the ring of polynomials over a noncommutative ring $R$ is defined by fixing the position of the coefficients, e.g., on the right: 
\[
f(x) = a_0+ x a_1+\ldots + x^n a_n;
\]
and by imposing commutativity of the indeterminate with the coefficients when such a polynomial is multiplied by another one, say $g(x) = b_0+ x a_1+\ldots + x^l b_l$:
\begin{equation}\label{introproduct}
(f \cdot g)(x) := a_0b_0 + x (a_0b_1+a_1b_0) + \ldots + x^{n+l}(a_n b_l).
\end{equation}
In general, the evaluation of this product at $y \in R$ is different from the product 
\[f(y)g(y) = a_0b_0 + a_0(yb_1)+(ya_1)b_0+\ldots + (y^na_n)(y^lb_l)\]
of the respective evaluations at $y$ (which we may call the pointwise product of $f$ and $g$ at $y$). However, a direct computation (based on the associativity of the multiplication in $R$) shows that: if $f(x)=0$ then $(f\cdot g)(x) =0$; if $f(x)$ is invertible, then
\begin{equation}\label{introassociativeproduct}
(f\cdot g)(x)=f(x)g(f(x)^{-1}xf(x)).
\end{equation}

In~\cite{cras,advances}, the authors constructed a class of quaternionic functions that includes the ring of polynomials just described (with $R=\hh$). They used the fact that for all $J \in \s_\hh := \{I \in \hh : I^2=-1\}$ the real subalgebra $\cc_J$ generated by $1$ and $J$ is isomorphic to $\cc$ and they decomposed the algebra into such ``slices'':
\[\hh=\bigcup_{J \in \s}\cc_J.\] 
On an open set $\OO\subseteq\hh$, they defined a differentiable function $f:\OO\rightarrow\hh$ to be (Cullen or) \emph{slice regular} if, for each $J\in\s$, the restriction of $f$ to $\OO_J:=\OO\cap\cc_J$ is a holomorphic function from $\OO_J$ to $\hh$, both endowed with complex structures defined by left multiplication by $J$. This definition comprises all polynomials and convergent power series of the form 
\begin{equation}\label{intropolynomials}
\sum_{n \in \nn} x^n a_n,\quad\{a_n\}_{n \in \nn} \subset \hh.
\end{equation}
At later stages the theory was endowed with a multiplicative operation, which generalizes~\eqref{introproduct} and verifies~\eqref{introassociativeproduct}; the zero sets were studied, proving to be the union of isolated points or isolated $2$-spheres of a special type; and the skew field of quotients of slice regular functions was constructed. This solid algebraic structure allowed to enlighten other features of slice regularity, which recall the theory of holomorphic complex functions. On these matters, we refer the reader to~\cite{librospringer}, which also points out the original references.

The definition of slice regular function carries over to the algebra of octonions $\oo$, see~\cite{rocky};~\cite{ghiloni} is another work specifically devoted to the octonionic case. The article~\cite{clifford} addressed the question for the Clifford algebra $\R_3$ and~\cite{israel} treated the case of functions from $\rr^{m+1}$ to the Clifford algebra $\R_m$, defining the notion of \emph{slice monogenic} function. The basic material on slice monogenicity is collected in~\cite{librodaniele2}, which also points out the original references. Finally,~\cite{perotti} introduced a new approach that is valid on the whole class of alternative $^*$-algebras over $\rr$ and comprises the cases just mentioned. The theory has then developed in subsequent literature, along with its applications to problems arising from other fields of mathematics: the definition of a functional calculus on such algebras (see ~\cite{librodaniele2,QSFCalculus}); the construction and classification of orthogonal complex structures on open dense subsets of $\rr^4\simeq \hh$,~\cite{ocs}; and the definition of sectorial operators over associative $^*$-algebras,~\cite{semigroups}.

The approach introduced in~\cite{perotti} for an alternative $^*$-algebra $A$ over $\rr$ makes use of the complexified algebra $A \otimes_{\R} \cc$, denoted by $A_{\cc}$. Let us denote its elements as $a+\ui b$, where $a,b \in A$ and $\ui$ is a special notation for the imaginary unit of $\cc$ designed to distinguish it from the $i$ that appears as an element of $A$ in some instances (e.g., $A=\cc,\hh,\oo$). Consider for simplicity the case of a quaternionic polynomial $f$ of the form~\eqref{intropolynomials}. For each $J \in \s_\hh$, the restriction $f:\cc_J \to \hh$ can be lifted through the map $\phi_J: \hh_\cc \to \hh, \phi_J(a+\ui b):=a+Jb$ and it turns out that the lift does not depend on $J$. In other words, there exists a polynomial function $F: \cc \simeq \rr_\cc \to \hh_\cc$ which makes the following diagram commute for all $J \in \s_\hh$:
\begin{equation}\label{lifting}
\begin{CD}
\rr_\cc @>F> >\hh_\cc\\ 
@V V \phi_J V 
@V V \phi_J V\\
\hh @>f> >\hh 
\end{CD}
\end{equation}
If another quaternionic polynomial $g$ lifts to a function $G$, then the product~\eqref{introproduct} lifts to the pointwise product $x \mapsto F(x)G(x)$ in $\hh_\cc$. Actually, the lift can be performed for any slice regular quaternionic function $f$. The resulting $F$ is holomorphic from $\rr_\cc$ to $\hh_\cc$ (both endowed with complex structures defined by left multiplication by $\ui$) and it belongs to the so-called class of \emph{stem functions}. This interpretation suggested to consider the entire class of \emph{slice functions}, namely those functions $f$ that admit lifts of type~\eqref{lifting} to stem functions $F$ (not necessarily holomorphic). We may describe this class as the most general family of functions compatible with the ``slice'' nature of $\hh$. The \emph{slice product}  $f\cdot g$ of two slice functions $f,g$ that lift to $F,G$ is then defined as the function that lifts to the pointwise product $FG$. This is a consistent generalization of~\eqref{introproduct} and of the multiplicative operation defined among slice regular functions. Moreover, it verifies formula~\eqref{introassociativeproduct}.

Now, instead of $\hh$ let us consider another alternative $^*$-algebra $A$ over $\rr$ (see Subsection~\ref{sec:algebras} for definition, properties and examples). An $A$-valued function is defined to be a slice function if it lifts to an $A_\cc$-valued stem function in a diagram similar to~\eqref{lifting}. If, moreover, its lift is holomorphic, then it is called a slice regular function. This is the key idea of the construction performed in~\cite{perotti}, which we will overview in full detail in Subsection~\ref{sec:functions}. The effectiveness of working with slice functions without restricting to the slice regular ones can be appreciated by looking at some of its applications, for instance the Cauchy integral formula for slice functions of class $\mscr{C}^1$,~\cite[Theorem~27]{perotti}, or the definition of a continuous (slice) functional calculus for normal operators in quaternionic Hilbert spaces~\cite{QSFCalculus}.

It is in this generality that we presently study the fundamental algebraic properties of the class of slice functions and of the subclass of slice regular functions. The paper is organized as follows. 

In Section~\ref{sec:preliminaries}, as we already mentioned, we present some preliminary material and we make some assumptions, which we will suppose valid throughout the paper. First, we recall the definition of real alternative $^*$-algebra and some basic properties of such an algebra. In particular, we study the connections between the invertibility of elements and the $^*$-involution of the algebra. We give several examples to describe the variety of algebraic phenomena that can occur, in both the associative and nonassociative settings: the division algebras $\cc,\hh,\oo$; the Clifford algebras $C\ell(p,q)$ (including $\rr_m= C\ell(0,m)$ but also the algebras of split complex numbers $\s\cc$ and of split quaternions $\s\hh$); the algebra of dual quaternions $\D\hh$; and the algebra of split octonions $\s\oo$. Then, we overview in detail the definitions of slice function and of slice regular function, along with examples and with the basic operations on these functions.
 
In Section~\ref{sec:reciprocal}, we look at the set of slice functions, itself, as an alternative $^*$-algebra. Building on the foundations prepared in the previous section, we determine the multiplicative inverse $f^{-\punto}$ of a slice function $f$. This is a more delicate matter than it was for slice regular quaternionic functions (see~\cite[\S5.1]{librospringer} and references) or slice monogenic functions (see~\cite[\S2.6]{librodaniele2} and references).

Section~\ref{sec:explicit_formulas} presents explicit formulas for the algebraic operations on slice functions. In particular, we extend formula~\eqref{introassociativeproduct} and its analog for quotients $f^{-\punto} \cdot g$ to slice functions $f,g$ over any associative $^*$-algebra. Both formulas were known only for slice regular functions in the quaternionic case (see~\cite[Theorem 3.4 and Proposition 5.32]{librospringer}).

In Section~\ref{sec:zeros}, we give quite an exhaustive description of the zero sets of slice and of slice regular functions. Several new, sometimes surprising, phenomena appear in the results and we are able to give examples of all of them in some of the aforementioned algebras. We also give complete characterizations of the zero sets of slice functions over $\rr_3$ and over $\s\oo$. At the end of the section, we focus on slice regular functions, generalizing the nice properties of the zeros known for slice regular functions over the quaternions  (see~\cite[Chapter 3]{librospringer} and references therein) and the octonions (see~\cite{rocky,ghiloni}), as well as for slice monogenic functions (see~\cite[\S2.5]{librodaniele2} and references). Indeed, we prove a significantly new version of the identity principle, which (in addition to its intrinsic interest) completes the study of the multiplicative inverse performed in Section~\ref{sec:reciprocal}. We also determine, for a slice regular function $f$, conditions that guarantee that its zero set consist of isolated points or isolated ``spheres'' of a special type. 

Finally, Section~\ref{sec:zerosofproducts} studies the effect of slice multiplication on the zero sets. It begins with a general result and some pathological examples. Then it presents a stronger characterization, illustrated by examples, under an appropriate compatibility hypothesis between the algebra under consideration and its $^*$-involution. This compatibility is automatically verified on associative $^*$-algebras, for which we provide a complete characterization. We point out that this applies, in particular, to all Clifford algebras $\rr_m$. At the end of the paper, we focus again on slice regular functions and we remove the extra hypotheses on the algebra. In this context, we are able to improve our general result linking the zeros of a product $f\cdot g$ to the zeros of $f$ and $g$ (cf.~\cite[\S3.2]{librospringer} and~\cite[\S3]{ghiloni} for the quaternionic and octonionic cases). We also determine sufficient conditions to exclude that a slice regular $f$ be a zero divisor in the algebra of slice functions.

\vspace{.3em}
\noindent{\bf Acknowledgements.} This work is supported by GNSAGA of INdAM and by the grants FIRB ``Differential Geometry and Geometric Function Theory" and PRIN ``Variet\`a reali e complesse: geometria, topologia e analisi armonica" of the Italian Ministry of Education.
We warmly thank the anonymous referee, whose precious suggestions have significantly improved our presentation.
\vspace{-.8cm}


\tableofcontents

\newpage


\section{Preliminaries}\label{sec:preliminaries}
\subsection{Alternative $^*$-algebras over $\rr$}\label{sec:algebras}

In the present section, we overview the definition of alternative $^*$-algebras over the real field $\rr$, some of their properties and some finite-dimensional examples. Furthermore, we make some assumptions that will hold throughout the paper.

\vspace{.5em}

{\bf Assumption (A).} \emph{Let $A$ be an \emph{alternative algebra} over $\rr$; that is, a real vector space endowed with a bilinear multiplicative operation such that the associator $(x,y,z)=(xy)z-x(yz)$ of three elements of $A$ is an alternating function.}

\vspace{.5em}

The alternating property is equivalent to requiring that $(x,x,y) = 0 = (y,x,x)$ for all $x,y \in A$, which is automatically true if $A$ is associative. Alternativity yields the so-called Moufang identities:
\begin{eqnarray}
(xax)y &=& x(a(xy))\label{moufang1}\\
y(xax) &=& ((yx)a)x\label{moufang2}\\
(xy)(ax) &=& x(ya) x.\label{moufang3}
\end{eqnarray}
It also implies \emph{power-associativity}: for all $x \in A$, we have $(x,x,x)=0$, so that the expression $x^n$ can be written unambiguously for all $n \in \nn$. Another important property of alternative algebras is described by the following result of E. Artin (cf.~\cite[Theorem 3.1]{schafer}): the subalgebra generated by any two elements of $A$ is associative.

We recall that $A$ is called a \emph{division algebra} if, for all $a,b \in A$ with $a \neq 0$, each of the equations $ax=b, xa=b$ admits a unique solution $x \in A$. This cannot be the case if $A$ admits \emph{zero divisors}; that is, nontrivial solutions $x$ to $ax=0$ or $xa=0$ with $a \neq 0$.

\vspace{.5em}

{\bf Assumption (B).} \emph{All algebras and subalgebras are assumed to be \emph{unitary}, that is, to have a multiplicative neutral element $1$; and $\rr$ is identified with the subalgebra generated by $1$.}

\vspace{.5em}

With this notation, $\R$ is always included in the \emph{nucleus} $\mathfrak{N}(A):=\{r \in A\, |\, (r,x,y)=0\ \forall\, x,y \in A\}$ and in the \emph{center} $\mathfrak{C}(A):=\{r \in \mathfrak{N}(A)\, |\, rx=xr\ \forall\, x\in A\}$ of the algebra $A$. We refer to~\cite{schafer} for further details on alternative algebras.

\vspace{.5em}

{\bf Assumption (C).} \emph{The algebra $A$ is assumed to be equipped with a \emph{$^*$-involution} $x \mapsto x^c$; that is, a (real) linear transformation of $A$ with the following properties: $(x^c)^c=x$ for every $x \in A$, $(xy)^c=y^cx^c$ for every $x,y \in A$ and $x^c=x$ for every $x \in \R$.}

\vspace{.5em}

For every $x \in A$, the \emph{trace} $t(x)$ of $x$ and the (squared) \emph{norm} $n(x)$ of $x$ are defined as
\begin{equation}\label{traceandnorm}
t(x):=x+x^c \quad \text{and} \quad n(x):=xx^c.
\end{equation}

Let us take a quick look at some examples. More details about them can be found in~\cite{ebbinghaus}.

\begin{examples}
The only examples of finite-dimensional alternative division algebras over $\rr$ are: the real field $\rr$, the complex field $\cc$, the skew field of quaternions $\hh$, and the algebra $\oo$ of octonions. On all such algebras $A$, conjugation is defined to act as $(r+v)^c=r-v$ for all $r \in \rr$ and all $v$ in the Euclidean orthogonal complement of $\rr$ in $A$. The norm $n(x)$ turns out to coincide with the squared Euclidean norm $\Vert x\Vert^2$. Hence, for all $x\neq0$, $n(x)$ is a nonzero real number that coincides with $n(x^c)$ and allows the explicit computation of $x^{-1}$ as $x^{-1} = n(x)^{-1} x^c = x^c\, n(x)^{-1}$. 
\end{examples}

For more general real alternative $^*$-algebras, things are not as simple: they may contain zero divisors and even nonzero elements $x$ with $n(x)=0$. In the latter case, the algebra is called \emph{singular}. We now recall the basic features of a well-known class of associative algebras. For a complete treatment, see~\cite{GHS}.

\begin{examples}
The Clifford algebra $C\ell_{p,q}$, sometimes denoted also by $\rr_{p,q}$, is the associative algebra that can be constructed by taking the space $\rr^{2^{n}}$ with $n=p+q$ and with the following conventions:
\begin{itemize}
\item $1, e_1,\ldots,e_n, e_{12}, \ldots, e_{n-1,n}, e_{123},\ldots,e_{1\ldots n}$ denotes the standard basis of $\rr^{2^{n}}$; 
\item $1$ is defined to be the neutral element;
\item $e_i^2 := 1$ for all $i \in \{1,\ldots,p\}$ and $e_i^2 := -1$ for all $i \in \{p+1,\ldots,n\}$;
\item for all $i_1,\ldots,i_s \in \{1,\ldots,n\}$ with $i_1<\ldots<i_s$, the product $e_{i_1}\ldots e_{i_s}$ is defined to be $e_{i_1\ldots i_s}$
\item $e_ie_j = -e_j e_i$ for all distinct $i,j$.
\end{itemize}
The algebra $C\ell_{p,q}$ becomes a $^*$-algebra when endowed with Clifford conjugation $x\mapsto x^c$, defined to act on $e_{i_1\ldots i_s}$ as the identity $id$ if $s\equiv0,3 \mod 4$ and as $-id$ if $s\equiv1,2 \mod 4$. 

\noindent In the special case in which $p=0$, the following properties hold for $\rr_n := C\ell_{0,n}$:
\begin{itemize}
\item $\rr_0 = \rr, \rr_1=\cc$ and $\rr_2 = \hh$; in particular, all these algebras have $n(x) = \Vert x\Vert^2$ and they are nonsingular;
\item for all $n\geq3$, denoting by $\langle \cdot,\cdot \rangle$ the Euclidean scalar product in $\rr_n$, it holds
\begin{equation}\label{cliffordnorm}
n(x) = \sum_{s\, \equiv\, 0,3 \text{ $\mathrm{mod}$ }4} \langle x, e_{i_1\ldots i_s}x \rangle e_{i_1\ldots i_s} = \Vert x\Vert^2 + \langle x, e_{123}x \rangle e_{123} + \ldots
\end{equation}
so that $\rr_n$ is always nonsingular; on the other hand, it is not a division algebra since for all nonzero $q \in \rr_2 \subseteq \rr_n$, the numbers $q\pm qe_{123}$ are zero divisors.
\end{itemize}

\noindent Here are a few examples of singular Clifford algebras:
\begin{itemize}
\item in the algebra $\s\cc = C\ell_{1,0}$ of split-complex numbers, $x=x_0+x_1e_1$ has norm $n(x) = (x_0+x_1e_1)(x_0-x_1e_1) = x_0^2-x_1^2$ so that $n^{-1}(0)$ is the union of the lines $x_0=x_1$ and $x_0=-x_1$;

\item in the algebra $\s\hh= C\ell_{1,1}$ of split-quaternions $x=x_0+e_1x_1+e_2 x_2 + e_{12}x_{12}$, the set $n^{-1}(0)$ is defined by the equation $x_0^2+x_2^2 = x_1^2 +x_{12}^2$. $\s\hh$ is  isomorphic to $C\ell_{2,0}$.
\end{itemize}
\end{examples}

Recent work on polynomials over the split quaternions includes~\cite{janovskaopfer}. Let us give another example of associative singular algebra, which is the object of current research for its applications to robotics. (See, e.g.,~\cite{hegedus} and references therein; we will include further information in Examples~\ref{ex:lowerdimensional}).

\begin{example}\label{dq}
The algebra of dual quaternions, denoted by $\D\hh$, can be defined as $\hh+\epsilon\hh$ where (for all $p,q \in \hh$) $(\epsilon p)(\epsilon q) = 0, p(\epsilon q) = \epsilon (p q) = (\epsilon p) q$. In particular, $\epsilon$ commutes with every element of $\D\hh$ and $\epsilon^2=0$.
Setting $(p+\epsilon q)^c=p^c+\epsilon q^c$ turns $\D\hh$ into a $^*$-algebra where $n(p+\epsilon q) = n(p)+\epsilon t(pq^c)$. Thus the set $n^{-1}(0)$ consists of the elements $\epsilon q$ with $q\in\hh$.
By analogy, we will denote the commutative $^*$-subalgebras $\rr+\epsilon \rr$ and $\cc+\epsilon \cc$ as $\D\rr$ and $\D\cc$, respectively.
\end{example}

Finally, let us consider an example of nonassociative singular algebra, which is also used to describe specific motions (see~\cite{baezhuerta}):

\begin{example}\label{so}
The real algebra of split-octonions $\s\oo$ can be constructed as $\hh+l\hh$ where (for all $p,q \in \hh$) $(lp)(lq) = q p^c, p(lq) = l(p^c q), (lp) q = l(qp)$. Setting $(p+lq)^c=p^c-lq$ turns $\s\oo$ into a $^*$-algebra where $n(p+lq) = (p+lq)(p^c-lq)=n(p)-n(q)+l (-p^cq+p^cq) = n(p)-n(q)$. Thus the set $n^{-1}(0)$ consists of all $p+lq$ with $p,q \in \hh$ having $\Vert p \Vert^2 = n(p) = n(q) = \Vert q \Vert^2$.
\end{example}

Going back to general alternative algebras, the following properties will turn out to be useful in the sequel.

\begin{lemma}\label{nonassociative}
Within an alternative algebra over $\rr$ (not necessarily of finite dimension), for all elements $x,y$:
\begin{enumerate}
\item if $x$ is invertible then $(x^{-1},x,y)=0$;
\item if $x,y$ are invertible then so is the product $xy$ and $(xy)^{-1} = y^{-1}x^{-1}$;
\item if the product $xy$ is invertible then $(x,y,(xy)^{-1}) = 0$; $y(xy)^{-1}$ is a right inverse for $x$; and $(xy)^{-1}x$ is a left inverse for $y$;
\item if the products $xy,yx$ are both invertible then $x,y$ are invertible as well and $x^{-1} = y(xy)^{-1} = (yx)^{-1}y, y^{-1}=(xy)^{-1}x = x (yx)^{-1}$.
\end{enumerate}
\end{lemma}

\begin{proof} The proofs are standard manipulations in nonassociative algebra:
\begin{enumerate} 
\item See~\cite[page 38]{schafer}.
\item For invertible $x,y$, property {\it 1} implies $xy = x (yx^{-1})x$. Hence, 
\[(y^{-1}x^{-1})(xy) = (y^{-1}x^{-1})(x (yx^{-1})x) = (((y^{-1}x^{-1})x)(yx^{-1}))x = (y^{-1}(yx^{-1}))x =1,\] 
where the second equality follows from the second Moufang identity~\eqref{moufang2} and the third and fourth equalities follow, again, from property 1. The equality $(xy)(y^{-1}x^{-1})=1$ follows from $(y^{-1}x^{-1})(xy) =1$ by substituting $y^{-1}$ for $x = (x^{-1})^{-1}$ and $x^{-1}$ for $y = (y^{-1})^{-1}$.
\item If $xy$ is invertible then property 1 implies $x=(xy)^{-1} (xyx)$. Hence, 
\[x ( y (xy)^{-1} ) = ( (xy)^{-1} (xyx) ) ( y (xy)^{-1} ) = (xy)^{-1} (xyxy) (xy)^{-1} = (xy)^{-1} (xy) = 1,\]
where the second equality follows from the third Moufang identity~\eqref{moufang3} while the third equality follows, again, from property 1. This proves that $y(xy)^{-1}$ is a right inverse for $x$ or, equivalently, that $(x,y,(xy)^{-1}) = 0$. As a consequence, $((xy)^{-1}, x,y) = 0$ so that $((xy)^{-1}x)y =(xy)^{-1}(xy) =1$. Hence $(xy)^{-1}x$ is a left inverse for $y$.
\item If the products $xy,yx$ are both invertible then (by property 3) $y(xy)^{-1}$ is a right inverse for $x$ and $(yx)^{-1}y$ is a left inverse for $x$. It turns out that $y(xy)^{-1} = (yx)^{-1}y$ as a consequence of the alternating property: indeed, $(yx)y=y(xy)$. The same reasoning applies if we swap $x$ and $y$.\qedhere
\end{enumerate}
\end{proof}

The previous lemma immediately implies the next property.

\begin{proposition}\label{invertibles}
Within an alternative $^*$-algebra over $\rr$ (not necessarily of finite dimension), an element $x$ is invertible if, and only if, $x^c$ is. If this is the case, then $(x^c)^{-1} = (x^{-1})^c$. Furthermore, $x$ is invertible if, and only if, both $n(x)$ and $n(x^c)$ are. If this is the case, then:
\[
x^{-1} =x^c\, n(x)^{-1} = n(x^c)^{-1}\, x^c.
\]
Moreover, $(n(x))^{-1} = (x^{-1})^c\, x^{-1} = n((x^{-1})^c)$.
\end{proposition}

A priori, if we only know that $n(x)$ is invertible then all we can say is that $x^c\, n(x)^{-1}$ is a right inverse for $x$ and $n(x)^{-1}\,x$ is a left inverse for $x^c$. 
In the special case when $n(x),n(x^c)$ both are nonzero real numbers, the previous proposition allows the explicit construction of the inverse of $x$. More in general, if $n(x),n(x^c)$ are in the center of the algebra $A$ then some nice properties hold. Let us set
\begin{eqnarray*}
N_A&:=&\{0\} \cup \big\{x \in A \, \big| \, n(x),n(x^c) \in \R^*\big\}\\
C_A&:=&\{0\} \cup \big\{x \in A \, \big| \, n(x),n(x^c) \text{ are invertible elements of the center of } A\big\}.
\end{eqnarray*}
These real cones are called the \emph{normal} and \emph{central cone} of $A$, respectively.

\begin{theorem}\label{centralproduct}
Let $A$ be an alternative $^*$-algebra over $\rr$ (not necessarily of finite dimension).
For each $x \in C_A\supseteq N_A$:
\begin{enumerate}
\item $n(x) = n(x^c)$;
\item if $x \neq 0$ then $x$ is invertible and $x^{-1} = n(x)^{-1}x^c$;
\item $(x,x^c,y) = 0$ for all $y \in A$;
\item $n(xy) = n(x) n(y) = n(y) n(x) = n(yx)$ for all $y \in C_A$.
\end{enumerate}
As a consequence, $C_A$ and $N_A$ are closed under multiplication and $C_A^* := C_A \setminus \{0\}, N_A^*: = N_A \setminus \{0\}$ are multiplicative loops.
\end{theorem}

\begin{proof}
All properties listed are obvious when $x=0$, so we restrict to the case when $n(x),n(x^c)$ are invertible elements of $A$ belonging to its center.
\begin{enumerate}
\item[1-2.] According to Proposition~\ref{invertibles}, the element $x$ is invertible, too. Since $xn(x) = n(x)x= x x^c x = x n(x^c)$, we conclude that $n(x) = n(x^c)$. The formula for $x^{-1}$ follows, again, from Proposition~\ref{invertibles}.
\item[3.] By Lemma~\ref{nonassociative} we have
\[0=(x,x^{-1},y) = (x, x^cn(x)^{-1},y)= n(x)^{-1} (x,x^c,y),
\]
whence $0=(x,x^c,y)$.
\item[4.] By the third Moufang identity~\eqref{moufang3}, $n(x) y x= x(x^c y) x$ so that
\[
n(x) n(yx) =n(x) (yx)(yx)^c = (x(x^c y) x) (yx)^c.
\]
Hence, by the first Moufang identity~\eqref{moufang1} and by property {\it 3},
\begin{align*}
n(x) n(yx) &= x((x^c y) (x (yx)^c)) = x((x^c y) (x (x^c y^c))) \\
&= x((x^c y) (n(x) y^c))  = n(x)  x((x^c y) y^c).
\end{align*}
Since $y \in C_A$, property {\it 3} ensures that  $(x^c y) y^c = x^c n(y)$ and hence
\[
n(x) n(yx) =  n(x)  x(x^c n(y)) =n(x)^2 n(y),
\]
which implies $n(yx) = n(x) n(y) = n(y) n(x)$ by the invertibility of $n(x)$.
\end{enumerate}
Finally, owing to property {\it 4}, if $x,y \in C_A^*$, then the numbers $n(xy)=n(x)n(y)$ and $n((xy)^c) = n(y^cx^c) = n(y^c)n(x^c)$ are still invertible elements of the center of $A$. Moreover, by Proposition~\ref{invertibles} and property {\it 2}, we have that $n(x^{-1})=n((x^{-1})^c)=n(x)^{-1}$. It follows that $xy,x^{-1} \in C_A^*$. Similarly, if $x,y\in N_A^*$, then $n(xy)=n((xy)^c),n(x^{-1})=n((x^{-1})^c) \in \rr^*$ and hence $xy,x^{-1} \in N_A^*$. Clearly, $N_A,C_A$ are both closed under left or right multiplication by $0$, which concludes the proof.
\end{proof}

Before proceeding, we make an additional remark for the associative case:

\begin{remark}\label{associativecasenorm}
If $A$ is associative then, for all $x,y\in A$, the equality $n(xy) = x y y^c x^c = x n(y) x^c$ holds. If $n(y)$ belongs to the center of $A$, then
\[
n(xy) = n(x) n(y) = n(y) n(x).
\]
\end{remark}

The associativity hypothesis can actually be weakened to the notion we are about to introduce.

\begin{definition}
We will call the alternative $^*$-algebra $A$ \emph{compatible} if the trace function $t$ defined by formula~\eqref{traceandnorm} has its values in the nucleus of $A$. 
\end{definition}

We point out that, according to the previous definition, every associative $^*$-algebra is compatible. Furthermore, some useful properties hold on compatible algebras, including a $^*$-analog of Artin's theorem:

\begin{theorem}
If $A$ is compatible, then:
\begin{enumerate}
\item the norm function $n$, as well, takes values in the nucleus of $A$;
\item for all $x,y\in A$, $(x,x^c,y)=0$.
\end{enumerate}
As a consequence, the $^*$-subalgebra generated by any two elements $x,y \in A$ is associative.
\end{theorem}

\begin{proof}
Point {\it 1} is a consequence of the equality $t(n(x))=2n(x)$, valid for all $x \in A$.
Point {\it 2} can be proven by direct computation. Indeed, for all $x,y\in A$, 
\[(x,x^c,y)=(x,x^c,y)-(x,t(x),y)=-(x,x,y)=0.\]
As for the final statement, it is obtained by applying a general result,~\cite[Theorem I.2]{bruckkleinfeld}, to the subsets $\{x,x^c\},\{x,x^c\},\{y,y^c\}$ of our alternative algebra $A$: since (by the alternativity hypothesis and by point {\it 2}) we have
\begin{align*}
0&=(x,x,z)=(x,x^c,z)=(x^c,x,z)=(x^c,x^c,z)\\
0&=(y,y,z)=(y,y^c,z)=(y^c,y,z)=(y^c,y^c,z)
\end{align*}
for all $z \in A$, the cited theorem guarantees the set $\{x,x^c,y,y^c\}$ is contained in an associative subalgebra of $A$.
\end{proof}

As a consequence, Remark~\ref{associativecasenorm} applies verbatim to every compatible $^*$-algebra $A$:

\begin{proposition}\label{compatiblecasenorm}
If $A$ is compatible then, for all $x,y\in A$, the equality $n(xy) = x n(y) x^c$ holds. If $n(y)$ belongs to the center of $A$, then
\[
n(xy) = n(x) n(y) = n(y) n(x).
\]
\end{proposition}

In order to keep our presentation as self-contained as possible, we also provide a direct proof of the same result.
\begin{proof}
The compatibility of $A$ and the third Moufang identity~\eqref{moufang3} imply that
\begin{align*}
n(xy)&=(xy)(y^cx^c)=(xy)(y^c(t(x)-x))=(xy)(y^ct(x)-y^cx)\\
     &=((xy)y^c)t(x)-(xy)(y^cx)=xn(y)t(x)-xn(y)x\\
     &=xn(y)x^c
\end{align*}
for all $x,y \in A$. As a consequence, $n(xy)=n(x)n(y)=n(y)n(x)$ if $n(y)$ belongs to the center of $A$.
\end{proof}

All the examples of alternative $^*$-algebras over $\rr$ that we have listed are compatible. We can construct here a non compatible example.

\begin{example}\label{soincompatible}
The algebra $\s\oo=\hh+ l\hh$ of split-octonions (cf.~Example~\ref{so}) admits also the $^*$-involution $(p+lq)^c=p^c+lq$. In particular, $t(l)=2l$ does not belong to the nucleus of the algebra. In the sequel, unless otherwise stated, we will always assume that $\s\oo$ is equipped with the  $^*$-algebra structure defined in Example~\ref{so}.
\end{example}

Going back to the general setting, if we want to deal with elements of $A$ that split into (uniquely determined) real and imaginary parts just as in the complex case, we may restrict to the \emph{quadratic cone} of $A$,
\[
Q_A:=\R \cup \big\{x \in A \, \big| \, t(x) \in \R ,n(x) \in \R, 4n(x)>t(x)^2\big\}.
\]
Indeed, it turns out that if we set
\begin{equation} \label{eq:s_A}
\s_A=\{x \in A \, | \, t(x)=0, n(x)=1\},
\end{equation}
then the unitary subalgebra generated by any $J \in \s_A$, i.e., $\cc_J=\langle 1,J \rangle$, is isomorphic to the complex field and
\begin{equation} \label{eq:slice}
\textstyle
\text{$Q_A=\bigcup_{J \in \s_A}\cc_J$.}
\end{equation}
Moreover, $\cc_I \cap \cc_J=\R$ for every $I,J \in \s_A$ with $I \neq \pm J$. In other words, every element $x$ of $Q_A \setminus \R$ can be written as follows: $x=\alpha+\beta J$, where $\alpha \in \R$ is uniquely determined by $x$, while $\beta \in \R$ and $J \in \s_A$ are uniquely determined by $x$, but only up to sign. If $x \in \R$, then $\alpha=x$, $\beta=0$ and $J$ can be chosen arbitrarily in $\s_A$. Therefore, it makes sense to define the \emph{real part} $\re(x)$ and the \emph{imaginary part} $\im(x)$ by setting $\re(x):=t(x)/2=(x+x^c)/2$ and $\im(x):=x-\re(x)=(x-x^c)/2$. Finally, for all $J \in \s_A$, we have that $J^c=-J$ so that the isomorphism between $\cc_J$ and $\cc$ is also a $^*$-algebra isomorphism. In particular, if $x=\alpha+\beta J$ for some $\alpha,\beta \in \rr$ and $J \in \s_A$, then $x^c=\alpha-\beta J$ and $n(x) = n(x^c)= \alpha^2+\beta^2$. Hence, for such an $x$ we may write $|x|:=\sqrt{n(x)}=\sqrt{\alpha^2+\beta^2}$ just as we would for a complex number. Moreover, $Q_A \subseteq N_A\subseteq C_A$.
We refer the reader to~\cite[\S 2]{perotti} for a proof of the preceding assertions.

For all finite-dimensional (alternative) division algebras $A$ over $\rr$, we have $Q_A =N_A =C_A=A$, but things are more complicated in general:

\begin{examples}\label{ex:lowerdimensional}
Let $\mathfrak{Z} = \mathfrak{Z}(A)$ denote the set of zero divisors in $A$. Then:

\vskip 8pt 
\begin{tabular}{|c|c|c|c|c|c|c|c|c|
}
\hline
$A$ & $\dim$ 
& {\rm nucleus}
& {\rm center} & $\mathfrak{Z}$& $C_A$ & $N_A$ & $Q_A$&$\s_A$ \\
\hline
$\rr $ & $1$
& $\rr$& $\rr$ & $\emptyset$ & $\rr$ & $\rr$ & $\rr$ &$\emptyset$ \\
$\cc$ & $2$
& $\cc$& $\cc$ & $\emptyset$ & $\cc$ & $\cc$ &  $\cc$ & $S^0$ \\
$\s\cc$ & $2$ 
& $\s\cc$& $\s\cc$ & $n^{-1}(0)^*$ & $\s\cc\setminus \mathfrak{Z}$ & $\s\cc\setminus \mathfrak{Z}$ &  $\rr$ & $\emptyset$ \\
$\D\rr$ & $2$ 
& $\D\rr$ & $\D\rr$ & $n^{-1}(0)^*$ & $\D\rr \setminus \mathfrak{Z}$ & $\rr$ & $\rr$ & $\emptyset$ \\
$\hh $ & $4$ 
& $\hh$ & $\rr$ & $\emptyset$ & $\hh$ & $\hh$ & $\hh$ &$S^2$ \\
$\s\hh$ & $4$ 
& $\s\hh$& $\rr$ & $n^{-1}(0) ^*$ & $\s\hh\setminus \mathfrak{Z}$ & $\s\hh\setminus \mathfrak{Z}$ &$\rr \cup U^4$ & $H^2$ \\
$\D\cc$ & $4$ 
& $\D\cc$ & $\D\cc$ & $n^{-1}(0)^*$ & $\D\cc \setminus \mathfrak{Z}$ & $\study^3\setminus \mathfrak{Z}$ & $\cc$ & $S^0$ \\
$\oo$ & $8$
& $\rr$ & $\rr$ & $\emptyset$  & $\oo$ & $\oo$ & $\oo$ &$S^6$ \\
$\s\oo$ & $8$
& $\rr$ & $\rr$ &  $n^{-1}(0)^*$ & $\s\oo \setminus \mathfrak{Z}$& $\s\oo \setminus \mathfrak{Z}$ & $\rr \cup U^8$ &$H^6$\\
$\D\hh$ & $8$ 
& $\D\hh$ & $\D\rr$ & $n^{-1}(0)^*$ & $\D\hh \setminus \mathfrak{Z}$ & $\study^7\setminus \mathfrak{Z}$ & $\mathscr{Q}^6\setminus \mathfrak{Z}$ & $\mathscr{T}^4$ \\
$\rr_3$ & $8$ 
& $\rr_3$ & $\Gamma^2$ & $\Delta^4$ & $\rr_3 \setminus \mathfrak{Z}$ & $E^7$ & $F^6$ & $G^4$ \\
\hline
\end{tabular}\vskip 8pt

\noindent In the scheme, we have used the following notations.
\begin{itemize}
\item $S^0,S^2,S^6$ are the unit spheres in the Euclidean subspaces $\im(\cc),\im(\hh),\im(\oo)$ respectively.
\item Among the split-quaternions $w=x_0+x_1e_1 + x_2e_2 + x_{12}e_{12}\in \s\hh$, the set $H^2$ is the (two-sheets) $2$-hyperboloid 
\[
x_0=0,\ x_2^2 = 1 + x_1^2 +x_{12}^2,
\]
which is asymptotic to the circular $2$-cone $x_2^2 = x_1^2 +x_{12}^2$ obtained by intersecting $n^{-1}(0)$ with the $3$-space $x_0=0$. The line $\rr$ added to the interior of the same $2$-cone forms the quadratic cone $\rr \cup U^4$, where $U^4$ is the open subset of $\s\hh$ defined by equation $x_2^2>x_1^2+x_{12}^2$.
\item Within the algebra $\D\cc$, of ``dual complex numbers'' $z+\epsilon w$ (with $z,w \in \cc$), the set $\study^3$ has equation $t(zw^c)=0$. In other words, it is the set of numbers $z+\epsilon w$ with mutually orthogonal $z,w$.
\item Among the split-octonions $p+lq\in \s\oo$, the set $H^6$ is defined by the equations
\[
t(p)=0,\ n(p) = 1 + n(q).
\]
while the quadratic cone is the open set $U^8$ of equation $n(Im(p))>n(q)$, completed with $\rr$.
\item Within the algebra $\rr_3$ of Clifford numbers $x_0+\sum_{1\leq l\leq3} x_l e_l + \sum_{1\leq h<k\leq3} x_{hk} e_{hk}+ x_{123}e_{123}$, the set $\Gamma^2$ is the plane $\rr+e_{123}\rr$ (isomorphic to $\s\cc$ as an algebra but not as a $^*$-algebra, since $e_{123}^c = e_{123}$). $\Delta^4$ is the disjoint union of the punctured $4$-spaces $(1+e_{123})\rr_2^*$ and $(1-e_{123})\rr_2^*$. $E^7$ is the real algebraic set 
\[
x_0x_{123}-x_1x_{23}+x_2x_{13}-x_3x_{12}=0.
\]
$F^6$ is the intersection between $E^7$ and the hyperplane $x_{123}=0$. Finally, $G^4$ is $E^7$ intersected with the unit sphere of $x_0=x_{123}=0$ (see~\cite{clifford} and~\cite[Example 1(3)]{perotti}).
\item Among the dual quaternions $p+\epsilon q\in \D\hh$, where $p=p_0+p_1i+p_2j+p_3k$, $q=q_0+q_1i+q_2j+q_3k$, the set $\study^7$ is defined by equation $t(pq^c)=0$, i.e., it is the {Study quadric}  \[p_0q_0+p_1q_1+p_2q_2+p_3q_3=0.\]
When the homogeneous set $\study^7\setminus \mathfrak{Z}$ with $\mathfrak{Z} = n^{-1}(0)^* = \epsilon \hh^*$ is considered in $\rr\mathbb{P}^7$, it gives a group isomorphic to the group $SE(3)$ of rigid motions of $\rr^3$: $x=p+\epsilon q\in \study^7\setminus \mathfrak{Z}$ corresponds to the isometry $\im(\hh)\ni v\mapsto pvp^{-1}+qp^{-1}$ of $\im(\hh)\simeq\rr^3$ (see e.g.~\cite{hegedus}). The set $\mathscr{Q}^6$ is the intersection of $\study^7$ with the hyperplane $q_0=0$. Both the normal cone $\study^7\setminus \mathfrak{Z}$ and the quadratic cone $\mathscr{Q}^6\setminus \mathfrak{Z}$ are not algebraic subsets of $\D\hh \simeq \rr^8$, but only semialgebraic.  Finally, $\mathscr{T}^4$ is the real algebraic set with equations
\[
p_0=q_0=0,\  p_1^2+p_2^2+p_3^2=1,\ p_1q_1+p_2q_2+p_3q_3=0.
\]
Hence, $\mathscr{T}^4$ can be seen as the tangent bundle over $S^2\subseteq \im(\hh)$.
\end{itemize}
\end{examples}

In all the aforementioned examples, $C_A$ is an open dense subset of $A$, namely $A \setminus \mathfrak{Z}(A)$. Furthermore, $A^* \setminus \mathfrak{Z}(A)$ coincides with the group of invertible elements in $A$, as it is the case in any finite-dimensional associative algebra:

\begin{proposition}\label{prop_farenick}
In a finite-dimensional associative algebra $A$, the only elements $x$ that are not invertible are $0$ and the zero divisors. In particular, such $x$ are neither right- nor left-invertible, nor are the products $xy, yx$ by any other $y \in A$ invertible.
\end{proposition}
For a proof, see~\cite[Theorem 2.24]{farenick}. To complete the panorama, let us consider the higher-dimensional Clifford algebras $\rr_m=C\ell_{0,m}$.

\begin{example}
The center of $\rr_m$ is $\rr$ for $m$ even and $\rr+\rr e_{1\ldots m}$ for $m$ odd. Now let $m\geq 3$. Taking into account formula~\eqref{cliffordnorm} and the fact that $e_{1\ldots 4n-1}^2=1$ while $e_{1 \ldots 4n+1}^2=-1$, the central cone coincides with the normal cone if and only if $m\not\equiv3\mod 4$.
The set $\s_{\rr_m}$ includes $e_{i_1\ldots i_s}$ if and only if $s\equiv1,2 \mod 4$. The quadratic cone $Q_{\rr_m}$ is a real algebraic proper subset of $\rr_m$. It always includes the so-called space of paravectors $\rr^{m+1} = \{x_0+\sum_{1\leq l\leq m} x_l e_l  \, \big| \,  x_0,x_l\in \rr\}$ and, in general, all vector spaces
\[
V_s = \left\{\left.x_0+\sum_{1\leq i_1<\ldots <i_s\leq m} x_{i_1\ldots i_s} e_{i_1\ldots i_s}  \, \right| \,  x_0,x_{i_1\ldots i_s}\in \rr\right\}
\]
with $s \in \{1, \ldots,m\}$ and $s \equiv 1 \mod 4$. Since $\dim V_s = 1+ \binom{m}{s} \geq 1+ \left(\frac m s\right)^s$ and since it is always possible to choose an $s \equiv 1 \mod 4$ between $\frac m 2 -2$ and $\frac m 2 +2$ (namely, $s =4\lfloor \frac{m+2} 8\rfloor +1$), we infer at once that $\dim Q_{\rr_m}$ grows exponentially in $m$. The following scheme holds for $n\geq1$:

\vskip 8pt 
\begin{tabular}{|l|l|l|l|c|l|l|l|
}
\hline
$A$ & $\dim$
& {\rm nucleus}

& {\rm center} & $n^{-1}(0)$ & $\mathfrak{Z}_A$ & $C_A$ & $Q_A$ 
\\
\hline
$\rr_{4n-1}$ & $2^{4n-1}$ 
& $\rr_{4n-1}$& $\rr+e_{1 \ldots 4n-1}\rr$ & $\{0\}$ & $\neq \emptyset$ & $\supsetneq N_A$ & $\supseteq V_{4\lfloor n /2 \rfloor+1}$
\\
$\rr_{4n}$ & $2^{4n}$ 
& $\rr_{4n}$& $\rr$ & $\{0\}$ & $\neq \emptyset$ & $=N_A$ &  $\supseteq V_{4\lfloor n /2 \rfloor+1}$
\\
$\rr_{4n+1}$ & $2^{4n+1}$
& $\rr_{4n+1}$& $\rr+e_{1\ldots 4n+1}\rr$ & $\{0\}$ & $\neq \emptyset$ & $=N_A$&  $\supseteq V_{4\lfloor n /2 \rfloor+1}$ 
\\
$\rr_{4n+2}$ & $2^{4n+2}$
& $\rr_{4n+2}$& $\rr$ & $\{0\}$ & $\neq \emptyset$ & $=N_A$ &  $\supseteq V_{4\lfloor (n+1)/2 \rfloor+1}$ 
\\
\hline
\end{tabular}\vskip 8pt

\end{example}


\subsection{Slice functions and slice regular functions}\label{sec:functions}

In this section, we overview some material from~\cite{perotti,expansionsalgebras}; namely, the definition of two special classes of $A$-valued functions and the description of their basic properties. Part of the definitions and results we are about to list (namely, those regarding slice regular functions) had been developed on specific algebras with a completely different approach. For the case $A= \hh$, we refer to~\cite[Chapter 1]{librospringer}, which also points out the corresponding references. For $A = \rr_3,\oo$, see~\cite{clifford,rocky}, respectively. Furthermore,~\cite{israel} had introduced the related theory of slice monogenic functions from $\rr^{m+1}$ to $\rr_m$. In the general setting, $A$-valued functions with domains contained in the quadratic cone $Q_A$ are studied. Therefore, it is natural to take the next assumption.

\vspace{.5em}

{\bf Assumption (D).} \emph{From this point on, we suppose $A$ to have finite dimension over $\rr$ and we assume that $Q_A \neq \R$ or, equivalently, that $\s_A \neq \emptyset$}.

\vspace{.5em}

For instance, the algebras $\rr, \s\cc, \D\rr$ are excluded.

Since left multiplication by an element of $\s_A$ induces a complex structure on $A$, we infer that the real dimension of $A$ equals $2h+2$ for some non-negative integer $h$. Then the real associative *-algebra $A$ has the following useful splitting property: for each $J \in \s_A$, there exist $J_1,\ldots,J_h \in A$ such that $\{1,J,J_1,JJ_1,\ldots,J_h,JJ_h\}$ is a real vector basis of $A$, called a \emph{splitting basis} of $A$ associated to $J$. 

\vspace{.5em}

{\bf Assumption (E).} \emph{Let $A$ be equipped with the natural $\mscr{C}^1$-manifold structure determined by the global coordinate systems associated with the real vector bases of $A$. We call the underlying topology the \emph{Euclidean topology} of $A$.}

\vspace{.5em}
The relative topology on each $\cc_J$ with $J \in \s_A$ clearly agrees with the topology determined by the natural identification between $\cc_J$ and $\cc$.
Given a subset $E$ of $\cc$, its \emph{circularization} $\OO_E$ is defined as the following subset of $Q_A$:
\[
\OO_E:=
\left\{x \in Q_A \, \big| \, \exists \alpha,\beta \in \R, \exists J \in \s_A \mathrm{\ s.t.\ } x=\alpha+\beta J, \alpha+\beta i \in E\right \}.
\]
A subset of $Q_A$ is termed \emph{circular} if it equals $\OO_E$ for some $E\subseteq\cc$. For instance, given $x=\alpha+\beta J \in Q_A$ we have that
\[
\s_x:=\alpha+\beta \, \s_A=\{\alpha+\beta I \in Q_A \, | \, I \in \s_A\}
\]
is circular, as it is the circularization of the singleton $\{\alpha+i\beta\}\subseteq \cc$. We observe that $\s_x=\{x\}$ if $x \in \R$. On the other hand, for $x \in Q_A \setminus \rr$, the set $\s_x$ is obtained by real translation and dilation from $\s_A$. Such sets are called \emph{spheres}, because the theory has been first developed in the special case of division algebras $A=\cc,\hh,\oo$ where they are genuine Euclidean spheres (see Examples~\ref{ex:lowerdimensional}).
The next observation will prove useful in the sequel.
\begin{remark}\label{preservedsphere}
Fix $a \in C_A^*$. For all $J\in \s_A$, we have $a^{-1}Ja \in \s_A$. Indeed, by Theorem~\ref{centralproduct},
\begin{align*}
n(a^{-1}Ja) &= n(a^{-1})n(J)n(a) = n(J)=1,\\
t(a^{-1}Ja) &= t (n(a)^{-1} a^c Ja) = n(a)^{-1} (a^cJa-a^cJa)=0.
\end{align*}
As a consequence, the function $x \mapsto a^{-1}xa$ maps any sphere $\s_y$ (and any circular subset of $Q_A$) to itself.
\end{remark}

\vspace{.5em}

{\bf Assumption (F).} \emph{Let $\OO=\OO_D$ be the circularization of a non-empty subset of $\cc$, denoted by $D$, which is invariant under the complex conjugation $z=\alpha+i\beta \mapsto \overline{z}=\alpha-i\beta$. As a consequence, for each $J \in \s_A$ the slice $\OO_J:=\OO\cap\cc_J$ is equivalent to $D$ under the natural identification between $\cc_J$ and $\cc$.} 

\vspace{.5em}

The class of functions we consider is defined by means of the complexified algebra $A_{\cc}=A \otimes_{\R} \cc=\{x+\ui y \, | \, x,y \in A\}$ of $A$, endowed with the following product:
\[
(x+\ui y)(x'+\ui y')=xx'-yy'+\ui (xy'+yx').
\]
The algebra $A_{\cc}$ is still alternative: the equality $(x+\ui y, x+\ui y, x'+\ui y') =0$ can be proven by direct computation, taking into account that $(x,y,x') = - (y,x,x')$ and  $(x,y,y') = - (y,x,y')$ because $A$ is alternative. Moreover, $A_{\cc}$ is compatible, or associative, if and only if $A$ is. We shall denote by $\rr_\cc$ the real subalgebra $\rr+\ui\rr\simeq\cc$ of $A_\cc$, which is included in the center of $A_\cc$.
In addition to the complex conjugation $\overline{x+\ui y}=x-\ui y$, we may endow $A_\cc$ with a $^*$-involution $x+\ui y \mapsto (x+\ui y)^c := x^c+\ui y^c$, which makes it a $^*$-algebra. We point out that (regardless of $A$) the complexified $^*$-algebra $A_\cc$ is singular:

\begin{example}
For each $J \in \s_A$, the number $1+\ui J$ has norm $(1+\ui J) (1-\ui J) = 1+J^2 + \ui (J-J) = 0$.
\end{example}

Given $F=F_1+\ui F_2:D \lra A_\cc$ with $A$-components $F_1$ and $F_2$, the function $F$ is called a \emph{stem function} on $D$ if $F(\overline{z})=\overline{F(z)}$ for every $z \in D$ or, equivalently, if $F_1(\overline{z})=F_1(z)$ and $F_2(\overline{z})=-F_2(z)$ for every $z \in D$.

\begin{definition} \label{def:slice-function}
A function $f:\OO \lra A$ is called a \emph{(left) slice function on $\OO = \OO_D$} if there exists a stem function  $F=F_1+\ui F_2:D \lra A_\cc$ such that, for all $z=\alpha+i\beta \in D$ (with $\alpha,\beta \in \R$), for all $J \in \s_A$ and for $x=\alpha+\beta J$,
\[
f(x)=F_1(z)+JF_2(z),
\]
In this situation, we say that $f$ is induced by $F$ and we write $f=\I(F)$. If $F_1$ and $F_2$ are $\R$-valued, then we say that the slice function $f$ is \emph{slice preserving}. 
We denote by $\mc{S}(\OO)$ the real vector space of slice functions on $\OO$ and by $\mc{S}_\rr(\OO)$ the subspace of slice preserving functions.
\end{definition}
The  terminology used in the definition is justified by the following property (cf.~\cite[Prop.~10]{perotti}): the components $F_1$ and $F_2$ of a stem function $F$ are $\R$-valued if and only if the slice function $f=\I(F)$ maps every ``slice'' $\OO_J$ into $\cc_J$.

Let $f:\OO \lra A$ be a slice function. Given $y=\alpha+\beta J$ and $z=\alpha+\beta K$ in $\OO$ for some $\alpha,\beta \in \rr$ and $J,K \in \s_A$ with $J-K$ invertible, as a direct consequence of the definition of slice function, the following representation formula holds for all $x=\alpha+\beta I$ with $I \in \s_A$:

\begin{equation}\label{rep1}
f(x)=(I-K)\left((J-K)^{-1}f(y)\right)-(I-J)\left((J-K)^{-1}f(z)\right).
\end{equation}
In particular, for $K=-J$,
\begin{equation}\label{rep2}
f(x)=\frac12\left(f(y)+f(y^c)\right)-\frac{I}2\left(J\left(f(y)-f(y^c)\right)\right).
\end{equation}

Suppose $f=\I(F) \in \mc{S}(\OO)$ with $F=F_1+\ui F_2$. It is useful to define a function $\vs f:\OO \lra A$, called \emph{spherical value} of $f$, and a function $f'_s:\OO \setminus \R \lra A$, called \emph{spherical derivative} of $f$, by setting
\[
\vs f(x):=\frac{1}{2}(f(x)+f(x^c))
\quad \text{and} \quad
f'_s(x):=\frac{1}{2}\im(x)^{-1}(f(x)-f(x^c)).
\] 
The original article~\cite{perotti} and subsequent papers used the notations $v_s f$ and $\partial_s f$ instead of $\vs f$ and $f'_s$, respectively. However, our present work requires a more compact notation.
Observe that $\vs f(x)=F_1(z)$ and $f'_s(x)={\rm im}(z)^{-1} F_2(z)$ if $x=\alpha+\beta J,z=\alpha+\beta i$ and ${\rm im}(z) = \beta = \frac{z-\overline z} {2i}$ is the imaginary part-function on $\cc$.
 It follows that $\vs f$ and $f'_s$ are slice functions, constant on each sphere $\s_x\subseteq \OO$.  Moreover, since $F_1$ and $F_2$ can be obtained starting from $f$, we infer at once that each slice function $f$ is induced by a unique stem function $F$.
 
In general, the pointwise product $x \mapsto f(x)g(x)$ of slice functions $f=\I(F)$ and $g=\I(G)$ is not a slice function. 

\begin{example}
By direct inspection, the quaternionic function $\hh \ni x \mapsto i x \in \hh$ is not a slice function, but it is the pointwise product of the constant function $i$ and of the identity function $\mr{id}_\hh$, which are slice.
\end{example}

On the other hand, it is easy to verify that the pointwise product $FG$ of stem functions $F$ and $G$ is again a stem function. This fact suggested the following definition.

\begin{definition} \label{def:}
Given $f=\I(F),g=\I(G) \in \mc{S}(\OO)$, the \emph{slice product of $f$ and $g$} is defined as the slice function $f \cdot g:=\I(FG) \in \mc{S}(\OO)$.
\end{definition}

\begin{example}
The constant function $f \equiv i$ on $\hh$ lifts to the constant function $F\equiv i$ and $g=\mr{id}_\hh$ lifts to $G(z) = z$. Taking into account that $F(z)G(z) = zi$ for all $z \in \cc$, it turns out that $(f\cdot g)(x)=xi$.
\end{example}
The example confirms that, in general, $f \cdot g \neq fg$. However, if the slice function $f$ is slice preserving, then we have that $f \cdot g=fg=g \cdot f$. Furthermore, if $f$ and $g$ are both slice preserving, then $f \cdot g=fg$ is again slice preserving.

Finally, we have already observed that $A_\cc$ is a $^*$-algebra. We may thus associate to each stem function $F=F_1+\ui F_2$ the function $F^c:D \lra A_\cc$ mapping $z$ to $F^c(z):=(F(z))^c=F_1^c(z)+\ui F_2^c(z)$ (where $F_h^c(z):=(F_h(z))^c$ for each $h \in \{1,2\}$). It turns out that $F^c$ is a stem function as well. The function $FF^c$, mapping $z$ to $F(z) F^c(z)=F(z) F(z)^c= n(F(z))$, is a stem function and it holds: $FF^c=n(F_1)-n(F_2)+\ui t(F_1F_2^c)$. Hence, for $f=\I(F) \in \mc{S}(\OO)$ we may set $f^c=\I(F^c)$ and define the \emph{normal function} of $f$ in $\mc{S}(\OO)$ as
\[
N(f)=f \cdot f^c=\I(FF^c).
\] 

\begin{example}\label{ex:Delta}
For fixed $y \in Q_A$, the binomial $f(x) = x-y$ has $f^c(x) = x-y^c$ and its normal function $N(f)(x) = (x-y) \cdot (x-y^c)$ coincides with the (slice preserving) quadratic polynomial
\[
\Delta_y(x):=x^2-xt(y)+n(y).
\]
If $y' \in Q_A$, then $\Delta_{y'}=\Delta_y$ if and only if $\s_{y'}=\s_y$.
\end{example}

Within the class of slice functions, we consider a special subclass having nice properties that recall those of holomorphic functions of a complex variable. This is done by imposing holomorphy on stem functions, in an appropriate sense.
If $\OO = \OO_D$ is open, then each slice $\OO_J=\OO\cap\cc_J$ with $J \in \s_A$ is open in the relative topology of $\cc_J$; therefore, $D$ itself is open. In this case, we let $\mc{S}^0(\OO)$ and $\mc{S}^1(\OO)$ denote the real vector spaces of slice functions on $\OO$ induced by continuous stem functions and by stem functions of class $\mscr{C}^1$, respectively.
Suppose that $f = \I(F) \in \mc{S}^1(\OO)$. The derivative $\partial F/\partial \overline{z}:D \lra A_{\cc}$ with respect to $\overline{z}= \alpha - \beta\ui$, that is,
\[
\frac{\partial F}{\partial\overline{z}}:=\frac{1}{2}\left(\frac{\partial F}{\partial\alpha}+\ui\frac{\partial F}{\partial\beta}\right)
\]
and the analogous $\partial F/\partial z:D \lra A_{\cc}$ are still stem functions, which induce the slice functions $\partial f/\partial x^c:=\I(\partial F/\partial \overline{z})$ and  $\partial f/\partial x:=\I(\partial F/\partial z)$ on $\OO$.

\begin{definition} \label{def:slice-regularity} 
Let $\OO$ be open. A slice function $f \in \mc{S}^1(\OO)$ is called \emph{slice regular} if $\partial f/\partial x^c\equiv0$ in $\OO$. We denote by $\mc{SR}(\OO)$ the real vector space of slice regular functions on $\OO$.
\end{definition}

Slice regularity is naturally related to complex holomorphy, in the following sense. We recall the notation $\OO_J:=\OO\cap\cc_J$ valid for all $J \in \s_A$.
    
\begin{lemma}\label{splitting} 
Let $\OO$ be open.
Let $J \in \s_A$ and let $\{1,J,J_1,JJ_1,\ldots,J_h,JJ_h\}$ be an associated splitting basis of $A$. For $f \in \mc{S}^1(\OO)$, let $f_0,f_1,\ldots,f_h : \OO_J \to \cc_J$ be the $\mscr{C}^1$ functions such that $f_{|_{\OO_J}}=\sum_{\ell=0}^h f_\ell J_{\ell}$, where $J_0:=1$. Then $f$ is slice regular if, and only if, for each $\ell \in \{0,1,\ldots,h\}$, $f_\ell$ is holomorphic from $\OO_J$ to $\cc_J$, both equipped with the complex structure associated to left multiplication by $J$.
\end{lemma}

An important fact is, that slice regularity is closed under addition and slice multiplication: if $f$ and $g$ are slice regular, then $f+g$ and $f \cdot g$ are slice regular as well. Moreover, $f$ is slice regular if and only if $f^c$ is, owing to the fact that, for all stem functions $F : D \to A$, by direct inspection $\partial F^c/\partial \overline{z} \equiv 0$ is equivalent to $\partial F/\partial \overline{z} \equiv 0$. As a consequence, a slice regular $f$ has a slice regular normal function $N(f)$.

Here are some classical examples of slice regular functions.

\begin{example}
Every polynomial of the form $\sum_{m=0}^n x^ma_m = a_0+x a_1 + \ldots x^n a_n$ with coefficients $a_0, \ldots, a_n \in A$ is a slice regular function on the whole quadratic cone $Q_A$.
\end{example}

\begin{example}
If the $^*$-algebra $A$ is endowed with a norm $\Vert \cdot \Vert_A$ such that $\Vert x \Vert_A^2 = n(x)$ for all $x \in Q_A$, it was proven in~\cite[Theorem 3.4]{expansionsalgebras} that every power series of the form $\sum_{n \in \nn} x^n a_n$ converges on the intersection between $Q_A$ and a ball $B(0, R) = \{x \in A\ |\ \Vert x \Vert_A<R\}$. If $R>0$, then the sum of the series is a slice regular function on $Q_A \cap B(0,R)$.
\end{example}

We conclude this preliminary section considering again the domains of our functions.
If $\OO=\OO_D$ and $D$ are open, then $D$ can be decomposed into a disjoint union of open subsets of $\cc$, each of which either
\begin{enumerate}
\item intersects the real line $\rr$, is connected and preserved by complex conjugation; or
\item does not intersect $\rr$ and has two connected components $D^+,D^-$, switched by complex conjugation.
\end{enumerate}
Therefore, when $D$ is open we may assume without loss of generality that it fall within case 1 or case 2. In the former case, the resulting domain $\OO=\OO_D$ is called a \emph{slice domain} because each slice $\OO_J$ with $J \in \s_A$ is a domain in the complex sense (more precisely, it is an open connected subset of $\cc_J$). In case 2, we will call $\OO=\OO_D$ a \emph{product domain} as it is homeomorphic to the Cartesian product between the complex domain $D^+$ and the sphere $\s_A$.


\section{The algebra of slice functions and ``division''}\label{sec:reciprocal}

The construction undertaken in~\cite{perotti} that we overviewed in the previous section can be neatly summarized as follows. For fixed $D \subseteq \cc$ and $\Omega=\Omega_D$ (in accordance with Assumption (F)):

\begin{enumerate}
\item[1.] the stem functions $D \to A_\cc$ form an alternative $^*$-algebra over $\rr$ with pointwise addition $(F+G)(z) = F(z)+G(z)$, multiplication $(FG)(z) = F(z)G(z)$ and conjugation $F^c(z) = F(z)^c$;
\item[2.] besides the pointwise addition $(f,g) \mapsto f + g$, a multiplication $(f,g) \mapsto f \cdot g$ and a conjugation $f \mapsto f^c$ can be defined on the set $\mc{S}(\OO)$ of slice functions on $\OO$ to make the mapping $\I$ a $^*$-algebra isomorphism from the $^*$-algebra of stem functions on $D$ to $\mc{S}(\OO)$.
\end{enumerate}
If, additionally, $\OO$ (hence $D$) is open, then:
\begin{enumerate}
\item[3.] the continuous, $\mscr{C}^1$ and holomorphic stem functions on $D$ form $^*$-subalgebras of the one described in point 1;
\item[4.] $\mc{S}^0(\OO),\mc{S}^1(\OO)$ and $\mc{SR}(\OO)$ are $^*$-subalgebras of $\mc{S}(\OO)$ (obtained as the respective images through $\I$ of the subalgebras listed in point 3).
\end{enumerate}
This allows us to study ``division'' in $\mc{S}(\OO)$ by means of Proposition~\ref{invertibles}. We thus undertake in this new setting the construction that is valid for slice regular quaternionic functions,~\cite[\S 5.1]{librospringer}. See also~\cite[\S 2.6]{librodaniele2} on slice monogenic functions. In what follows, we take into account that the norm of an element $f$ of the alternative $^*$-algebra $\mc{S}(\OO)$ is the normal function $N(f)$ we already encountered in Section~\ref{sec:functions}.

\begin{theorem}\label{invertibles_f}
A slice function $f$ admits a multiplicative inverse $f^{-\punto}$ if, and only if, $f^c$ does. If this is the case, then $(f^c)^{-\punto} = (f^{-\punto})^c$. Furthermore, $f$ admits a multiplicative inverse $f^{-\punto}$ if, and only if, both $N(f)$ and $N(f^c)$ do. If this is the case, then: 
\[
f^{-\punto} =f^c\cdot N(f)^{-\punto} = N(f^c)^{-\punto}\cdot f^c.
\]
Moreover, $N(f)^{-\punto} = (f^{-\punto})^c\cdot f^{-\punto} = N((f^{-\punto})^c)$.
\end{theorem}

We already know that, in such a setting, an inverse can be explicitly constructed for an element $f$ if $f$ is in the normal cone of the real algebra $\mc{S}(\OO)$ and $f$ is not identically $0$, that is, if $N(f)$ and $N(f^c)$ are nonzero real constants. Fortunately, the real constants are not the only elements of the center of $\mc{S}(\OO)$ that can be easily inverted. 
For all $f \in \mc{S}(\OO)$, let us denote
\[
V(f):=\{x \in \OO \, | \, f(x)=0\}.
\]
We also recall that in Definition~\ref{def:slice-function} we set out the notation $\mc{S}_\rr(\OO)$ for the subalgebra of slice preserving functions.

\begin{lemma}\label{slicepreservinginverse}
Each slice preserving slice function $g \in \mc{S}_\rr(\OO)$ belongs to the center of $\mc{S}(\OO)$. Its zero set $V(g)$ is circular and, provided $\OO'=\OO \setminus V(g)$ is not empty, the following properties hold.
\begin{enumerate}
\item For all $x \in \OO'$, the values $g(x)$ belong to $Q_A^*$ and in particular they are invertible.
\item The function $h:\OO' \to A$, defined as $h(x)= g(x)^{-1}$, is a slice preserving function and it is the reciprocal $g^{-\punto}$ of $g$ in $\mc{S}(\OO')$.
\item If $\OO'$ is open in $Q_A$, then $g^{-\punto}$ is slice regular if and only if $g$ is slice regular in $\OO'$.
\end{enumerate}
If $\OO=\OO_D$ is a slice domain or a product domain and if $g$ is slice regular and not identically zero, then $V(g)$ consists of isolated real points or isolated spheres $\s_x$ and $\OO'=\OO \setminus V(g)$ is automatically an open dense subset of $\OO$.
\end{lemma}

\begin{proof}
In $A_\cc= A + \ui A$, the center includes $\rr_\cc = \rr+\ui \rr$ and the reciprocal of each nonzero element $x+\ui y \in \rr_\cc$ can be constructed with the usual formula $(x-\ui y)(x^2+y^2)^{-1}$. 
As a consequence, each $\rr_\cc$-valued stem function $G=G_1+\ui G_2:D \lra \rr_\cc$ associates and commutes with all other stem functions. Let $V(G)$ be its zero set, which must be discrete if $G$ is holomorphic and not identically zero. By the previous discussion, we have a well defined $G^{-1}: D \setminus V(G) \lra \rr_\cc$ (which is holomorphic if and only if $G$ is holomorphic in $D \setminus V(G)$).

Now, if $g$ is a slice preserving function, then $g= \I(G)$ for some $\rr_\cc$-valued stem function $G$. As a consequence, $g$ belongs to the center of $\mc{S}(\OO)$. Moreover, $V(g)$ is the circularization of $V(G)$ and $\OO'$ is the circularization of $D \setminus V(G)$. Setting $h = \I(G^{-1})$, it follows immediately that $h \cdot g = \I(G^{-1} G) = \I(1) =1$ and $g\cdot h = \I(GG^{-1}) =  \I(1) =1$ in $\OO'$. Furthermore, $h$ is slice regular if, and only if, $G^{-1}$ is holomorphic. Finally, a direct computation proves that $h(x)= g(x)^{-1}$ for all $x \in \OO'$, completing the proof.
\end{proof}

We point out that here and in the rest of the paper the restriction $g_{|_{\OO'}}$ is denoted again as $g$ and $g^{-\punto}$ stands for $\left(g_{|_{\OO'}}\right)^{-\punto}$. Similarly, the product $f\cdot g$ of two slice functions $f,g$ whose domains of definition intersect in a smaller domain $\tilde \OO\neq \emptyset$ should be read as $f_{|_{\tilde\OO}} \cdot g_{|_{\tilde\OO}}$. The consistency of this notation is guaranteed by the fact that slice multiplication is induced by the pointwise multiplication of the corresponding stem functions.

Lemma~\ref{slicepreservinginverse} exhibits a nontrivial class of functions in the center of $\mc{S}(\OO)$ admitting reciprocals. By means of Proposition~\ref{invertibles}, we can explicitly construct reciprocals for a larger subclass of $\mc{S}(\OO)$.

\begin{definition}
A slice function is termed \emph{tame} if $N(f)$ is a slice preserving function and it coincides with $N(f^c)$.
\end{definition}

Putting together Lemma~\ref{slicepreservinginverse}, Proposition~\ref{invertibles} and Theorem~\ref{centralproduct}:

\begin{proposition}\label{reciprocal}
If $f \in \mc{S}(\OO)$ is tame and $\OO' := \OO \setminus V(N(f))$ is not empty then $f$ belongs to the central cone $C_{\mc{S}(\OO')}$. In particular, $f$ is invertible in $\mc{S}(\OO')$; its reciprocal can be computed as
\[
f^{-\punto}(x) = (N(f)^{-\punto} \cdot f^c) (x)= (N(f)(x))^{-1} f^c(x);
\]
and, when $\OO'$ is open, $f^{-\punto}$ is slice regular if and only if $f$ is slice regular in $\OO'$.

Furthermore, the associator $(f,f^c,g)$ vanishes for all $g \in \mc{S}(\OO')$ and 
\begin{equation}\label{Nfg}
N(f \cdot g) = N(f) N(g) = N(g)\cdot N(f) = N(g\cdot f)
\end{equation}
for all $g$ in the central cone $C_{\mc{S}(\OO')}$. In particular, the slice product $f \cdot g$ of $f$ with any tame $g \in \mc{S}(\OO')$ is a tame element of $\mc{S}(\OO'')$ with $\OO'' := \OO' \setminus V(N(g))$ (provided $\OO''$ is not empty).
\end{proposition}

If we assume regularity for $f$, then we can say a bit more about the domain $\OO'$ of its reciprocal. Indeed, Lemma~\ref{slicepreservinginverse} implies what follows.

\begin{remark}\label{dense}
If $\OO=\OO_D$ is a slice domain or a product domain and if $f\in\mc{SR}(\OO)$ is tame, then either $N(f) \equiv 0$ and $\OO' = \OO \setminus V(N(f))$ is empty; or $V(N(f))$ consists of isolated real points or isolated spheres $\s_x$ and $\OO' = \OO \setminus V(N(f))$ is dense in $\OO$.
\end{remark}

However, at this point we have no instrument to exclude that $N(f) \equiv 0$. In Sections~\ref{sec:zeros} and~\ref{sec:zerosofproducts}, we will study the zeros of slice functions and we will gain some understanding about the existence of zero divisors among them, thus completing Proposition~\ref{reciprocal}.
Let us now add another result that will prove useful in the sequel.

\begin{proposition}\label{Nfgregular}
Assume that $\OO$ is a slice domain or a product domain.
If $f$ and $g$ are slice regular and tame on $\OO$, and $N(f)$ and $N(g)$ do not vanish identically, then $N(f\cdot g)=N(f)N(g)=N(g)N(f)$ on $\OO$. As a consequence, $f\cdot g$ is a tame element of $\mc{SR}(\OO)$.
\end{proposition}

\begin{proof}
By Proposition~\ref{reciprocal}, Equation~\eqref{Nfg} holds in $\OO''=\OO \setminus (V(N(f))\cup V(N(g)))$. According to our hypotheses and to Remark~\ref{dense}, $\OO''$ is dense in $\OO$. Therefore, by continuity, Equation~\eqref{Nfg} holds in $\OO$.
\end{proof}


\subsection{The associative or compatible case}

If $A$ is associative, then $A_\cc$ is associative, too. The algebra of stem functions $D \to A_\cc$ is therefore associative and so is the algebra $\mc{S}(\OO)$ of slice functions. The converse also holds, as the algebras of stem and slice functions include all $A$-valued constant functions. The same reasoning proves that $\mc{S}(\OO)$ is compatible if and only if $A$ is.

In this case, the multiplicative properties are a bit stronger as an immediate consequence of Proposition~\ref{compatiblecasenorm}:

\begin{remark}\label{compatibleN}
Suppose $A$ is compatible and $f,g \in \mc{S}(\OO)$. Then $N(f \cdot g) = f \cdot N(g) \cdot f^c$. If, moreover, $N(g)$ is in the center of $\mc{S}(\OO)$ then
\begin{equation}
N(f \cdot g) = N(f)\cdot N(g) = N(g)\cdot N(f).
\end{equation}
This is true, in particular, when $g$ is tame. As a consequence, if $f,g$ are both tame in $\OO$ then $f \cdot g$ is tame in $\OO$, too.
\end{remark}


\section{Explicit formulas for the operations on slice functions}\label{sec:explicit_formulas}

We now find some explicit formulas for multiplication, conjugation and normalization on $\mc{S}(\OO)$. In the associative case, we will also find an explicit formula for reciprocals and quotients. 

Let us begin with a technical lemma. Any juxtaposition $\im h$ denotes the function $x \mapsto \im(x) h(x)$ obtained by pointwise product.

\begin{lemma}
Consider a function $\mc{S}(\OO)$, its spherical value function $\vs  f \in \mc{S}(\OO)$ and its spherical derivative $f'_s \in \mc{S}(\OO \setminus \rr)$. Consider, moreover, the slice preserving function $\im:Q_A \to Q_A$. Then
\[
f=\vs  f + \im \cdot\, f'_s = \vs  f + \im  f'_s
\]
in $\OO\setminus \rr$. Furthermore, $f=\vs f$ in $\OO \cap \rr$.
\end{lemma}

\begin{proof}
If $F(z) = F_1(z) + \ui F_2(z)$ is the stem function inducing $f$, then we already observed that $\vs f(x)=F_1(z)$ and $ f'_s(x)={\rm im}(z)^{-1} F_2(z)$ for $x = \alpha+J\beta \in \OO \setminus \rr$, $z = \alpha+i \beta$, ${\rm im}(z) = \beta$, where $F_1(z), {\rm im}(z)^{-1} F_2(z)$ are $A$-valued stem functions, so that $\vs  f,f'_s$ are slice functions that are constant along each sphere $\s_x$. Moreover, 
\[
F(z) = F_1(z) + (\ui\, {\rm im}(z)) ({\rm im}(z)^{-1} F_2(z)),
\]
where $\ui \, {\rm im}(z)$ is itself a stem function with values in $\ui \rr$. The induced function $\im(\alpha+\beta J) = J \beta$ is therefore a slice preserving function. The thesis easily follows.
\end{proof}

By direct computation, we can make the following observation.

\begin{remark}
Let $f,g \in \mc{S}(\OO)$.
For all $x \in \OO \cap \rr$, we have $f^c(x)=f(x)^c$, $(f \cdot g)(x) = f(x) g(x)$ and $N(f)(x)=n(f(x))$.
On the other hand, for all $x \in \OO \setminus \rr$,
\begin{equation}\label{decomposedconjugate}
f^c(x) = \underbrace{\vs  f(x)^c}_{\vs{(f^c)}(x)}+ \im(x)   \underbrace{f'_s (x)^c}_{(f^c)'_s (x)}
\end{equation}
and
\begin{equation} \label{decomposednormal}
N(f)(x) =\underbrace{n(\vs  f(x)) + \im(x)^2 n( f'_s(x))}_{\vs{N(f)}(x)} + \im(x)\, \underbrace{t\big(\vs  f(x)\, f'_s(x)^c\big)}_{{N(f)}'_s(x)}.
\end{equation}
Moreover, the formula
\begin{equation}\label{decomposedproduct}
f\cdot g = \underbrace{\vs  f\,\vs  g+\im^2 f'_s\, g'_s}_{\vs{(f\cdot g)}}+\im\, \underbrace{(\vs  f\, g'_s+ f'_s\, \vs  g)}_{{(f\cdot g)}'_s}.
\end{equation}
holds in $\OO \setminus \rr$. In particular, for each $x \in \OO\setminus \rr$, if $g$ is constant along the sphere $\s_x$ (i.e., if $g'_s(x)=0$), then $g^c(x) = g(x)^c, N(g)(x) = n(g(x))$ and
\begin{equation}\label{productwithcircular}
(f\cdot g)(x) = f(x)g(x) - (\im(x), f'_s(x),g(x)).
\end{equation}
\end{remark}

We point out that, while $\im$ is in the center of $\mc{S}(\OO)$ (for all $g,h \in\mc{S}(\OO)$, the equalities $\im \cdot\, g= g \cdot \im$ and $\im \cdot\, (h \cdot g) = (\im \cdot\, h) \cdot g$ hold), the single values of $\im$ are not in the center nor in the nucleus of $A$ in general so that $(\im(x),h(x),g(x))$ needs not vanish. Therefore, while we may well write $\im \cdot\, h\, \cdot g$, the expression $\im h\, g$ would be ambiguous in a nonassociative setting. Additionally, since the function $\im$ is slice preserving, $(\im \cdot\, h\, \cdot g) (x) = \im(x) (h \cdot g)(x)$. Thus, the notation $\im\ (h\cdot g)$ can be used in place of $\im \cdot h \cdot g$.

In order to prove the announced formulas, we need one more technical remark.

\begin{remark}\label{imaginaryassociator}
For all $a, b \in A$ and all $J \in \s_A$
\begin{align*}
J(J, a,b) &= J((Ja)b -J(ab))\\
&= J((Ja)b) + ab\\
&= - (J,Ja,b) +(J^2 a) b + ab\\
&= - (J,Ja,b)
\end{align*}
Furthermore, $J(J, a,b) = - J(J, b,a) = (J,Jb,a) = -(J, a,Jb)$.
\end{remark}

\begin{theorem}
For all 
$f,g \in \mc{S}(\OO)$,
\begin{equation}\label{productwithdecomposition}
f\cdot g = f \cdot \vs  g+\im\, (f \cdot  g'_s)
\end{equation}
in $\OO \setminus \rr$. Moreover, for all $x \in \OO \setminus \rr$,
\begin{equation}\label{productwithassociators}
(f\cdot g)(x) = f(x) \vs  g(x)+\im(x) \big(f(x)  g'_s(x)\big) - (\im(x),  f'_s(x),g(x^c)).
\end{equation}
\end{theorem}

\begin{proof}
Formula~\eqref{productwithdecomposition} follows directly from the decomposition $g = \vs  g+\im \cdot\,  g'_s$ and from the fact that $f\cdot(\im \cdot\,  g'_s)= \im \cdot f \cdot  g'_s = \im\, (f\cdot  g'_s)$. Owing to formula~\eqref{productwithcircular},
\begin{align*}
(f \cdot \vs  g)(x) &= f(x) \vs  g(x) - (\im(x),  f'_s(x),\vs  g(x))\\
\im(x) (f \cdot  g'_s)(x) &= \im(x) \big(f(x)  g'_s(x)\big) - \im(x)(\im(x),  f'_s(x), g'_s(x)).
\end{align*}
By Remark~\ref{imaginaryassociator}, 
\begin{align*}
-& (\im(x),  f'_s(x),\vs  g(x)) - \im(x)(\im(x),  f'_s(x), g'_s(x))\\
=& - (\im(x),  f'_s(x),\vs  g(x)) + (\im(x), f'_s(x),\im(x) g'_s(x))\\
=&-\big(\im(x), f'_s(x),\vs  g(x)-\im(x) g'_s(x)\big)\\
=&-(\im(x), f'_s(x),g(x^c))
\end{align*}
and the proof of~\eqref{productwithassociators} is complete.
\end{proof}


\subsection{The associative case}
\label{Theassociativecase}

Within an associative algebra $A$, formula~\eqref{productwithassociators} simplifies significantly and it allows further manipulation.

\begin{corollary}\label{formula_associative_case}
If $A$ is associative then, for all $f,g \in \mc{S}(\OO)$ and for all $x \in \OO \setminus \rr$,
\[
(f\cdot g)(x) = f(x) \vs  g(x)+\im(x) f(x)  g'_s(x).
\]
Therefore:
\begin{itemize}
\item if $f(x) =0$ then $(f\cdot g)(x) =0$;
\item if $f(x)$ is invertible then $(f\cdot g)(x) = f(x) \big(\vs  g(x)+f(x)^{-1}\im(x) f(x)  g'_s(x)\big)$;
\item if $f(x)$ is invertible and $f(x)^{-1} x f(x)\in \s_x$ then $(f\cdot g)(x) = f(x)g\big(f(x)^{-1} x f(x)\big)$.
\end{itemize}
\end{corollary}

By Remark~\ref{preservedsphere}, $f(x) \in C_A^*$ guarantees that $f(x)^{-1} x f(x) = \re(x) + f(x)^{-1}\im(x) f(x)$ lies in $\s_x$. Here is another sufficient condition, valid under the present associativity assumption.

\begin{lemma} \label{lem:T_g}
Suppose $A$ is associative. Then, for all $f \in \mc{S}(\OO)$ and for all $x \in \Omega$ such that $f(x)$ is invertible, the point $f(x)^{-1} x f(x)$ belongs to $\s_x$, provided $n(f(x))$ commutes with $x$. The latter commutativity condition is always satisfied if $f$ is tame.
\end{lemma}

\begin{proof}
Let $x=\alpha+\beta J \in \Omega$ be such that $f(x)$ is invertible. Then $f(x)^{-1} x f(x) = \alpha + \beta J_1$ with $J_1:=f(x)^{-1} Jf(x)$. If $\beta=0$, the assertion is evident. Now suppose $\beta \neq 0$. Our first goal is to prove that, under the hypothesis that $n(f(x))$ commutes with $x$ (or, equivalently, with $J$), $J_1\in\s_A$. By direct computation,
\[
n(J_1)=-f(x)^{-1}Jn(f(x))J(f(x)^c)^{-1}=f(x)^{-1}n(f(x))(f(x)^c)^{-1}=1.
\]
Since $J_1^2=f(x)^{-1} J^2f(x)=-1$, it follows that $J_1^c=-J_1$. Hence, $t(J_1)=0$.

Our second goal is proving that $n(f(x))$ commutes with $x$ if $f$ is tame. In order to do so, we decompose the value $f(x) = f(\alpha+\beta J)$ as $a_1+Ja_2$, where $a_1 = \vs  f(x), a_2 = \beta  f'_s(x)$. By formula~\eqref{decomposednormal}, $N(f)(x) = n(a_1)-n(a_2)+J t(a_1a_2^c)$ where $n(a_1)-n(a_2)$ and $t(a_1a_2^c)$ are both slice preserving because $f$ is tame. Now,
\[
n(f(x)) = (a_1+Ja_2)(a_1^c-a_2^cJ) = n(a_1) - J n(a_2) J - a_1a_2^c J + J a_2 a_1^c
\]
so that
\[
Jn(f(x)) - n(f(x))J = J (n(a_1)-n(a_2)) - J t(a_1a_2^c)J - (n(a_1)-n(a_2)) J - t(a_1a_2^c) = 0.
\]
\end{proof}

Let us restate our result on the product of tame functions in the associative setting, adding a formula for reciprocals and quotients. Both formulas were previously known for slice regular quaternionic functions, which are always tame (see~\cite[Theorem 3.4 and Proposition 5.32]{librospringer}).

\begin{theorem} \label{fg-associative}
Suppose $A$ is associative. Let $f,g \in \mc{S}(\OO)$ with $f$ tame. Then
\begin{equation} \label{productformula}
(f \cdot g) (x) = f(x) g(f(x)^{-1} x f(x)),
\end{equation}
whenever $f(x)$ is invertible. Moreover, for all $x \in \OO\setminus V(N(f))$ such that $f^c(x)$ is invertible, if we set $T_f(x) := f^c(x)^{-1}x f^c(x)$ then $f(T_f(x))$ is invertible and
\begin{equation}
(f^{-\punto} \cdot g) (x) = f(T_f(x))^{-1} g(T_f(x)).
\end{equation}
Furthermore, the map $T_f:\{x \in \OO \, | \, f^c(x) \mbox{ is invertible}\} \lra \{x \in \OO \, | \, f(x) \mbox{ is invertible}\}$ is invertible and its inverse is $T_{f^c}$.
\end{theorem}

\begin{proof}
The first statement immediately follows from the previous results.
As for the second one, we have that $f^{-\punto} \cdot g= N(f^c)^{-\punto} \cdot f^c \cdot g$ in $\OO\setminus V(N(f))$. Since $f$ and $f^c$ are tame,
\begin{align*}
(f^{-\punto} \cdot g) (x) &= N(f^c)(x)^{-1} (f^c \cdot g)(x)\\
&= (f^c \cdot f)(x)^{-1} (f^c \cdot g)(x)\\
&= (f^c(x) f(T_f(x)))^{-1} f^c(x) g(T_f(x))\\
&= f(T_f(x))^{-1} f^c(x)^{-1} f^c(x) g(T_f(x))\\
&= f(T_f(x))^{-1} g(T_f(x)),
\end{align*}
whenever $f^c(x)$ is invertible. The invertibility of $f(T_f(x))$ follows from that of $N(f^c)(x) = f^c(x) f(T_f(x))$ and that of $f^c(x)$.

Finally, if $f^c(x)$ is invertible and $y:=T_f(x)$, then $f(y)$ is invertible and
\begin{align*}
T_{f^c}(T_f(x)) &=f(y)^{-1}yf(y)=f(y)^{-1}\left(f^c(x)^{-1}xf^c(x)\right)f(y)=\\
&=\left(f^c(x)f(y)\right)^{-1}x\left(f^c(x)f(y)\right).
\end{align*}
Now, $f^c(x)f(y)=f^c(x)f(T_f(x))=(f^c\cdot f)(x)=N(f^c)(x)$ where the value $N(f^c)(x)$ commutes with $x$ owing to the fact that $N(f^c)$ is slice preserving. Hence $T_{f^c}(T_f(x))=x$.
\end{proof}

   
\section{Zeros of slice and slice regular functions}\label{sec:zeros}

The zero sets of slice regular functions are extremely interesting on all finite dimensional division algebras. Indeed, the discreteness of the zeros of holomorphic complex functions yields a peculiar structure for the zero sets of slice regular quaternionic functions: see~\cite[Chapter 3]{librospringer} for a full account and appropriate references. The situation is even more manifold over the octonions,~\cite{ghiloni}. The new approach introduced in~\cite{perotti} makes it possible to address the case of other alternative $^*$-algebras $A$ and to generalize part of the results to slice functions that are not regular. While in~\cite{perotti} this was done under the additional hypothesis of \emph{admissibility} for $f \in \mc{S}(\OO)$ (that is, $\langle \vs f(x),f'_s(x) \rangle \subseteq N_A$ for all $x \in \OO \setminus \R$), we try to draw here a general picture, taking advantage of the representations we obtained in Section~\ref{sec:explicit_formulas}. At the end of the present section, we will focus again on slice regular functions.
   
\begin{theorem}\label{zeros}
Let $f \in \mc{S}(\OO)$ and let $x \in \OO\setminus \rr$.
\begin{enumerate}
           \item If $f'_s(x)=0$ then one of the following holds:
           \begin{enumerate}
               \item $\s_x$ is included in $V(f),V(f^c),V(N(f))$ and $V(N(f^c))$.
               \item $\s_x$ does not intersect $V(f)$ nor $V(f^c)$. Moreover, it does not intersect $V(N(f))$ nor $V(N(f^c))$, unless $A$ is singular and $\vs f(x)$ (which is the constant value of $f_{|_{\s_x}}$) or $\vs f(x)^c$ belong to $n^{-1}(0)^*$.
         \end{enumerate}
            \item If $f'_s(x)$ is a right zero divisor then one of the following holds:
           \begin{enumerate}
               \item $\s_x$ intersects $V(f)$ at at least one point $y$ and the intersection is the set of all $y' \in \s_x$ such that
               \begin{equation}\label{morezeros}
      (\im(y')-\im(y)) f'_s(x)= 0.
      \end{equation}
               Moreover, $y^c \not \in V(f)$ and $y^c \in V(N(f^c))$. The equality
               	\begin{equation}\label{associatorderivative}
(f'_s(x), f'_s(x)^c,y)=0
\end{equation}

is equivalent to $\s_x \subseteq V(N(f^c))$ and to $y \in V(N(f))$. If this is the case then $\s_x \subseteq V(N(f))$ if, and only if, the commutator $[n(f'_s(x)),y]$ vanishes.
 \item $\s_x$ does not intersect $V(f)$.
\end{enumerate}
 \item If $f'_s(x)$ is neither $0$ nor a right zero divisor then one of the following holds:           
\begin{enumerate}
 \item $\s_x$ intersects $V(f)$ at exactly one point $y$ and $y^c \in V(N(f^c))$.\\
 If $f'_s(x)$ is invertible, then $y=\re(x)-\vs f(x) f'_s(x)^{-1}$.\\
 Condition~\eqref{associatorderivative} is equivalent to $\s_x \subseteq V(N(f^c))$ and to $y \in V(N(f))$. If this is the case, then $\s_x \subseteq V(N(f))$ if, and only if, $[n(f'_s(x)),y]=0$.\\ In particular, if $f'_s(x)\in C_A^*$ then $\s_x$ is included both in $V(N(f))$ and in $V(N(f^c))$ and, additionally, $\s_x\cap V(f^c)= \{f'_s(x)^{-1}y^c f'_s(x)\}$.
\item $\s_x$ does not intersect $V(f)$. If $f'_s(x)\in C_A^*$, then $\s_x$ does not intersect $V(f^c)$ and it cannot be included both in $V(N(f))$ and in $V(N(f^c))$. If, moreover, $\vs f(x)\in C_A^*$ (or if $A$ is compatible), then $\s_x$ is not included in $V(N(f))$ nor in $V(N(f^c))$.
           \end{enumerate}
       \end{enumerate}
      Finally, if $x \in \OO \cap \rr$ then $\s_x=\{x\}$ has property {\it 1.(a)} or property {\it 1.(b)}.
   \end{theorem}

\begin{proof}
If $x\in\alpha+\beta\s_A$ with $\alpha, \beta \in \rr$ and $\beta \neq0$, then the value $f(\alpha+\beta I)$ can be decomposed as $a_1+Ia_2$ for all $I \in \s_A$ where $a_1 = \vs  f(x), a_2 = \beta  f'_s(x)$. Moreover, $f^c(\alpha+\beta I) = a_1^c+Ia_2^c$ and, according to formula~\eqref{decomposednormal}, $N(f)(\alpha+\beta I) = n(a_1)-n(a_2)+I t(a_1a_2^c), N(f^c)(\alpha+\beta I) = n(a_1^c)-n(a_2^c)+I t(a_1^ca_2)$ for all $I \in \s_A$.

If $f(y) =0$ with $y=\alpha + J \beta$ then $a_1 = -Ja_2$ so that (using the third Moufang identity~\eqref{moufang3}):
\begin{align*}
n(a_1)-n(a_2) &= -(Ja_2)(a_2^cJ) - n(a_2) = -Jn(a_2)J + JJ n(a_2) = J [J,n(a_2)]\\
t(a_1a_2^c) &= -(Ja_2)a_2^c + a_2(a_2^cJ)=-(J,a_2,a_2^c) - Jn(a_2) + n(a_2)J - (a_2,a_2^c,J)\\ 
& = -2(J,a_2,a_2^c)-[J,n(a_2)].
\end{align*}
Hence $N(f)(y)= n(a_1)-n(a_2)+Jt(a_1a_2^c) = -2J(J,a_2,a_2^c)$. Moreover, if the latter equals $0$, then $N(f)$ vanishes identically on $\s_x$ if, and only if, $[J,n(a_2)]=0$.
We also compute (using Remark~\ref{imaginaryassociator})
\begin{align*}
n(a_1^c)-n(a_2^c) &= -(a_2^cJ)(Ja_2) + a_2^c(JJa_2) = -(a_2^c,J,Ja_2) \\
&= -(J,Ja_2,a_2^c) = J(J,a_2,a_2^c)\\
t(a_1^ca_2) &= (a_2^cJ)a_2 -a_2^c(Ja_2) = (a_2^c,J,a_2) = (J,a_2,a_2^c).
\end{align*}
Thus, $N(f^c)(y^c) = n(a_1^c)-n(a_2^c) -J t(a_1^ca_2)=0$ and $N(f^c)$ vanishes identically on $\s_x$ if, and only if, $(J,a_2,a_2^c)=0$.

We now consider the three different possibilities for $a_2$.
       
\begin{enumerate}
\item If $a_2=0$ then $f(\alpha+\beta I) = a_1, f^c(\alpha+\beta I) = a_1^c, N(f)(\alpha+\beta I) = n(a_1)$ and $N(f^c)(\alpha+\beta I) = n(a_1^c)$ for all $I \in \s_A$. If $a_1=0$ then $\s_x$ is included in $V(f),V(f^c),V(N(f))$ and $V(N(f^c))$; if, on the contrary, $a_1\neq0$ then $f$ and $f^c$ never vanish in $\s_x$ (while $V(N(f))$ and $V(N(f^c))$ either include or do not intersect $\s_x$ depending on whether or not $n(a_1)=0$ and $n(a_1^c)=0$). 
\item If $a_2$ is a right zero divisor and $y=\alpha + J \beta$ is a zero of $f$ (so that $a_1=-Ja_2$) then $\s_x \cap V(f) = \{\alpha + I\beta \,|\, a_1 + Ia_2 = 0\}= \{\alpha + I\beta \,|\, (I-J)a_2 = 0\}$. This set cannot include $y^c=\alpha-J\beta$, as $-2Ja_2 \neq 0$.
\item If $a_2\neq 0$ is not a right zero divisor and $y=\alpha + J \beta$ is a zero of $f$ (so that $a_1=-Ja_2$), then $\s_x \cap V(f) = \{y\}$ since $(I-J)a_2 = 0 \Leftrightarrow I=J$. If, furthermore, $a_2$ is invertible then $J = -a_1 a_2^{-1}$ by direct computation. If $a_2$ is actually in $C_A^*$, then Theorem~\ref{centralproduct} and Remark~\ref{preservedsphere} apply and we can make use of the fact that, for all $J\in \s_A,a \in C_A^*$, we have $a^{-1}Ja \in \s_A, (a,a^c,J)=0, [J,n(a)]=0$:
\begin{itemize}
\item If $f(y) = a_1+Ja_2 = 0$ then for $K:= -a_2^{-1}Ja_2 \in \s_A$ we have $f^c(\alpha+\beta K) = a_1^c+Ka_2^c= a_1^c - a_2^{-1}Jn(a_2) = a_1^c - a_2^cJ = f(y)^c=0$. Moreover, in this case $(a_2,a_2^c,J)=0$ and $[J,a_2]=0$ imply that $N(f),N(f^c)$ vanish identically in $\s_x$ by the first part of the proof.
\item If $0=N(f)(\alpha+\beta I) =n(a_1)-n(a_2)+It(a_1a_2^c)$ for all $I \in \s_A$ then $n(a_1)=n(a_2)$ and $t(a_1a_2^c) = 0$. If $a_2 \in C_A^*$ then $U:=-a_1a_2^{-1} = -n(a_2)^{-1}a_1a_2^c$ has $t(U) = -n(a_2)^{-1}t(a_1a_2^c) = 0$. If, moreover, $n(U) = n(a_1) n(a_2)^{-1} = 1$ then automatically $U\in \s_A$ and $f(\alpha+\beta U) = a_1+Ua_2=0$. The desired equality $n(U) = n(a_1) n(a_2)^{-1}$ holds true if $A$ is compatible (see Proposition~\ref{compatiblecasenorm}) or if $a_1$ also belongs to $C_A$ (see case {\it 4} of Theorem~\ref{centralproduct}). The latter is the case if we put the additional hypothesis that $0 = N(f^c)(\alpha+\beta I) =n(a_1^c)-n(a_2^c)+It(a_1^ca_2)$ for all $I \in \s_A$.
\end{itemize}
\end{enumerate}

Finally, when $\beta=0$ and $x=\alpha \in \rr$ the computations are analogous to those performed in point 1.
\end{proof}

Example~\ref{ex:Delta} is an instance of case {\it 1.(a)}. Here is an example of the possible pathology in case {\it 1.(b)}. 

\begin{example}\label{ex:singular}
In $A=\s\oo$, the constant function $f \equiv 1+l$ belongs to $\mc{SR}(Q_A)\subseteq\mc{S}(Q_A)$. By direct computation, $N(f) = N(f^c) \equiv 0$, while $V(f) = V(f^c) = \emptyset$. A similar example can be constructed in any singular $^*$-algebra $A$.
\end{example}

In cases {\it 2.(a)}, {\it 3.(a)}, if $A$ is one of the algebras $\cc,\hh,\s\hh,\D\cc,\oo,\s\oo,\D\hh,\rr_3$, there cannot be examples with $\s_x \not \subseteq V(N(f))$ because the norm $n$ has its values in the center. However, such counterexamples exist in higher-dimensional Clifford algebras:

\begin{example}\label{ex:R4}
In $A=\rr_4$ (hence in all $\rr_m$ with $m\geq4$), for any $t \in \rr^*$ the first-degree polynomial $f(x) = (x-e_4)(1+te_{123})=x(1+te_{123})-{e_4}+t{e_{1234}}$ belongs to $\mc{SR}(Q_A)\subseteq\mc{S}(Q_A)$ and it has a zero at $e_4\in \s_A$. By direct computation, $n(1+te_{123})=(1+te_{123})^2 = 1+t^2+2te_{123}$. Hence
\[
N(f^c)(x)= (1+te_{123})\cdot (x+e_4)\cdot (x-e_4)\cdot (1+te_{123})= (x^2+1)(1+t^2+2te_{123})
\]
so that $\s_A \subseteq V(N(f^c))$, while
\begin{align*}
N(f)(x)&= (x-e_4)\cdot (1+t^2+2te_{123}) \cdot (x+e_4)\\
&= (1+t^2)(x^2+1) + 2t (x-e_4)\cdot(x-e_4) e_{123}\\ 
&= (x^2+1)(1+t^2+2te_{123})+ 4t (x-e_4)e_{1234}
\end{align*}
(where we have taken into account $(x-e_4)\cdot(x-e_4)= x^2-2xe_4-1 = x^2+1-2(x-e_4)e_4$). Thus, $\s_A\cap V(N(f)) = \{e_4\}$.

For the sake of completeness, we observe that $\s_A\cap V(f^c) = \emptyset$. Indeed, by Theorem~\ref{zeros}, $\s_A\cap V(f^c)$ must be included in $\s_A\cap V(N(f))^c= \{-e_4\}$ and it is easy to verify that $f^c(x) = x(1+te_{123})+e_4+te_{1234}$ does not vanish at $-e_4$.
\end{example}

The class of functions just mentioned, with $f$ and $f^c$ swapped, provides examples of cases {\it 2.(b)} and {\it 3.(b)} where the zero set of the normal function and of the conjugate function are not empty. Here is another pathological example of case {\it 2.(b)}.

\begin{example}\label{ex:conjugate}
In $A=\s\oo$, the function $f(x) = i - li + x (1- l )$ belongs to $\mc{SR}(Q_A)\subseteq\mc{S}(Q_A)$ and it has no zeros. Indeed, if it vanished at $x=q+lw$ with $q,w \in \hh$ then \[0 = i - li + (q+lw) (1- l )= i+q-w^c + l (-i-q^c +w)\] 
would imply the contradictory equalities $w=i+q^c,w^c=i+q$. On the other hand, $f^c(x) = -i + li + x (1+l )$ vanishes at $x=q+lw$ (with $q,w \in \hh$) if, and only if,
\[0 = -i + li + (q+lw) (1+l ) = q+w^c-i + l (q^c+w+i),\]
that is, if and only if $w=-i-q^c$. Therefore, $V(f^c) =Q_A \cap \{q-l(i+q^c)\, |\, q \in \hh\}$. By direct computation, $N(f)=N(f^c) \equiv 0$.
\end{example}

We complete the panorama with an example of case {\it 1.(b)} taking place at real points.

\begin{example}\label{ex:singular2}
Let $A = \s\hh$ or $A=\s\oo$ and let $a \in A$ have $n(a)=0$ while $a \neq 0$. Then for $f(x) = x-a$, the normal function $N(f)(x)=N(f^c)(x) = x^2 - x t(a)+0=x (x-t(a))$ vanishes at $0$ and at $t(a)$ while $f$ and $f^c$ have no zeros in $Q_A$.
\end{example}

We now come to some consequences of Theorem~\ref{zeros}.

\begin{corollary}\label{corollaryzeros}
Let $f \in \mc{S}(\Omega)$.
\begin{enumerate}
 \item The inclusion $V(N(f^c))^c \supseteq V(f)$ holds.
 \item If $f'_s(x) \in C_A$ for every $x \in \Omega \setminus \R$, then for every $x \in \Omega$ the sets $\s_x \cap V(f)$ and $\s_x \cap V(f^c)$ are both empty, both singletons, or both equal to $\s_x$. Moreover,
 \[
 V(N(f)) \supseteq \bigcup_{x \in V(f)}\s_x=\bigcup_{x \in V(f^c)}\s_x \subseteq V(N(f^c)).
 \]
 If, additionally, $A$ is nonsingular, then 
 \[
 \bigcup_{x \in V(f)}\s_x =\bigcup_{x \in V(f^c)}\s_x = \bigcup_{\s_x \subseteq V(N(f))\cap V(N(f^c))}\s_x,
 \]
whence $\bigcup_{x \in V(f)}\s_x=\bigcup_{x \in V(f^c)}\s_x=V(N(f))=V(N(f^c))$ if $f$ is also tame.
 \item Suppose that $f_s^\circ(x) \in C_A$ for every $x \in \Omega$ and $f'_s(x) \in C_A$ for every $x \in \Omega \setminus \R$. Then
 \[
\bigcup_{\s_x \subseteq V(N(f))}\s_x=\bigcup_{x \in V(f)}\s_x=\bigcup_{x \in V(f^c)}\s_x=\bigcup_{\s_x \subseteq V(N(f^c))}\s_x,
 \]
whence $V(N(f))=\bigcup_{x \in V(f)}\s_x=\bigcup_{x \in V(f^c)}\s_x=V(N(f^c))$ if $f$ is also tame and $V(N(f))=V(f)=V(f^c)=V(N(f^c))$ if $f$ is slice preserving.
\end{enumerate}
\end{corollary}

We now draw some conclusions in the compatible case, exploiting Theorem~\ref{zeros} and Proposition~\ref{compatiblecasenorm}. The statement is sharp because of the aforementioned examples.

\begin{corollary}\label{corollaryzeroscompatible}
Suppose $A$ is compatible and let $f \in \mc{S}(\Omega)$.
\begin{enumerate}
 \item  The inclusions
 \[
 V(N(f^c)) \supseteq \bigcup_{x \in V(f)}\s_x
 \quad \text{and} \quad
 V(N(f)) \supseteq V(f).
 \]
 hold.
 \item If $A$ is nonsingular and $f'_s(x) \in C_A$ for every $x \in \Omega \setminus \R$, then for every $x \in \Omega$ the sets $\s_x \cap V(f)$ and $\s_x \cap V(f^c)$ are both empty, both singletons, or both equal to $\s_x$. Moreover,
 \[
\bigcup_{\s_x \subseteq V(N(f))}\s_x=\bigcup_{x \in V(f)}\s_x=\bigcup_{x \in V(f^c)}\s_x=\bigcup_{\s_x \subseteq V(N(f^c))}\s_x.
 \]
\end{enumerate}
\end{corollary}

We now take advantage of the characterization of the zero divisors available for $\s\oo$ and we complete the description of the zero sets in this special case. Before stating our result, we point out that, if $A=\s\oo$, then $N(f)=N(f^c)$ for all $f \in \mc{S}(\OO)$, owing to the next remark.

\begin{remark}\label{uniquenormal}
If $n(a) = n(a^c)$ and $t(ab) = t(ba)$ for all $a,b \in A$ then for every slice function $f$ the equality $N(f)=N(f^c)$ holds as a consequence of formula~\eqref{decomposednormal}.
\end{remark}

\begin{proposition} \label{prop:SO}
If $A = \s\oo$, if $f \in \mc{S}(\OO)$ and if $x \in \OO$ then one of the following possibilities applies:
\begin{enumerate}
\item $V(f) \cap \s_x = \emptyset$;
\item $V(f) \cap \s_x= \{y\}$; $f'_s(x)\in C_A^*=\s\oo \setminus n^{-1}(0)$ and $y=\re(x)-\vs f(x) f'_s(x)^{-1}$;
\item $V(f) \cap \s_x$ is an affine $2$-plane in $\s\oo\simeq\rr^8$ and $f'_s(x)$ is a zero divisor (i.e., it belongs to $n^{-1}(0)^*$);
\item $V(f) \supseteq \s_x$ and $f'_s(x)=0$.
\end{enumerate}
In each of the aforementioned cases, respectively:
\begin{enumerate}
\item either $V(N(f)) \cap \s_x = \emptyset = V(f^c) \cap \s_x$; or $V(N(f)) \supseteq \s_x$, and $V(f^c) \cap \s_x$ is either empty or $2$-dimensional;
\item $V(N(f)) \supseteq \s_x$ and $V(f^c) \cap \s_x=\{f'_s(x)^{-1}y^c f'_s(x)\}$;
\item $V(N(f)) \supseteq \s_x$ and $V(f^c) \cap \s_x$ is either empty or $2$-dimensional; 
\item $\s_x$ is included both in $V(f^c)$ and in $V(N(f))$.
\end{enumerate}
\end{proposition}

\begin{proof}
We saw that a split-octonion $c+l d \in \s\oo = \hh + l \hh$ (with $c,d \in \hh$) is a zero divisor if, and only if, its norm $n(c+l d) = n(c)-n(d)$ vanishes (while $c,d \neq 0$). All other nonzero elements are invertible. If $c+l d$ is indeed a zero divisor then, by direct computation, the solutions $a+l b$ of
\begin{equation*}
0=(a+l b)(c+l d) = ac+db^c + l (a^cd+cb)
\end{equation*}
form the real $4$-space $S_{c,d} = \{a+l b \, |\, a \in \hh, b=-c^{-1}a^c d\}$.

Hence, solving~\eqref{morezeros} when $f'_s(x) = c+l d$ is a zero divisor is equivalent to finding those $y' \in \s_x$ such that $\im(y') - \im(y) \in S_{c,d}$.

We saw that if $A = \s\oo$ then
\begin{equation*}
\s_A = H^6 = \left\{w+l q \, | \,w \in \im(\hh), q \in \hh, n(w)-n(q)=1\right\}.
\end{equation*}
Therefore, $y$ decomposes as $y = \alpha + \beta (w + l q)$ for some $\alpha,\beta \in \rr, w \in \im(\hh), q \in \hh$ with $n(w)-n(q)=1$ and we only have to find those $L \in H^6$ such that $L- (w + l q) \in S_{c,d}$. These form the set
\begin{equation*}
H^6 \cap (w + l q + S_{c,d}) = \{(w+a)+l (q-c^{-1}a^cd) \, |\, a \in \im(\hh), \re(wa^c+q^cc^{-1}a^cd) = 0\},
\end{equation*}
where the function $a \mapsto \re(wa^c+q^cc^{-1}a^cd)$ from $\im(\hh)$ to $\rr$ is $\rr$-linear and is not identically $0$. Hence, $H^6 \cap (w + l q + S_{c,d})$ is a real affine $2$-plane. Along with Theorem~\ref{zeros}, this proves the first statement.

The second statement follows from Theorem~\ref{zeros} if we take into account the compatibility of $\s\oo$, along with two other properties valid in $\s\oo$. The first is, $(f'_s(x))^c \in n^{-1}(0)^*$ if, and only if, $f'_s(x) \in n^{-1}(0)^*$. The second is, that $t$ and $n$ are real valued, so that (by formula~\eqref{decomposednormal}) $N(f)$ is a slice preserving function.
\end{proof}

Examples~\ref{ex:singular} and~\ref{ex:conjugate} prove that Proposition~\ref{prop:SO} is sharp. Another special case, that of $\rr_3$, can be treated. In this case, too, Remark~\ref{uniquenormal} holds.

\begin{proposition} \label{prop:R_3}
If $A = \rr_3$, if $f \in \mc{S}(\OO)$ and if $x \in \OO$ then one of the following happens:
\begin{enumerate}
\item $V(f) \cap \s_x=\emptyset$;
\item $V(f) \cap \s_x=\{y\}$, $f'_s(x)\in C_A^*$ and $y=\re(x)-\vs f(x) f'_s(x)^{-1}$; 
\item $V(f) \cap \s_x=\{y,z\}$ for some $z \not \in \{y,y^c\}$ that commutes with $y$; $f'_s(x)$ is a zero divisor; and $(y-z)f'_s(x) = 0$; 
\item $V(f) \supseteq \s_x$ and $f'_s(x)=0$.
\end{enumerate}
In each of the aforementioned cases, respectively:
\begin{enumerate}
\item $\s_x$ does not intersect $V(f^c)$ nor $V(N(f))$;
\item $\s_x \subseteq V(N(f))$ and $V(f^c) \cap \s_x=\{f'_s(x)^{-1}y^c f'_s(x)\}$;
\item $\s_x \subseteq V(N(f))$ and $V(f^c) \cap \s_x=\{h^{-1}y^ch, h^{-1}z^ch\}$, where $h \in \rr_2^*$ is such that $f'_s(x)=(1\pm e_{123})h$; 
\item $\s_x$ is included both in $V(f^c)$ and in $V(N(f))$.
\end{enumerate}
\end{proposition}

\begin{proof}
If $x\in\alpha+\beta\s_{\rr_3}$, then $f(\alpha+\beta I)=a_1+Ia_2$ for each $I \in \s_{\R_3}$, where $a_1 = \vs  f(x), a_2 = \beta  f'_s(x)$.
By Theorem~\ref{zeros}, the statement concerning $V(f)$ will be proven if we establish that: when $a_2$ is a zero divisor $(1\pm e_{123})h$ (with $h \in \rr_2^*$), for all $u \in \s_{\rr_3}$ there exists a unique $v \in \s_{\rr_3} \setminus\{\pm u\}$ (which commutes with $u$) such that $(v-u)a_2= 0$.
We first suppose that $u = e_1$: $(v-e_1)a_2= 0$ holds if and only if $v-e_1 = (1\mp e_{123}) k$ for some $k \in \rr_2^*$, that is, $v = \mp e_{23}$.
More generally, any element $u \in \s_{\rr_3}$ can be obtained as $u = a^{-1}e_1a$ for some invertible $a \in\rr_3$. Using the same technique, we prove that $(v-u)a_2= 0$ if, and only if, $v =\mp a^{-1}e_{23}a$ and the thesis follows.

Now, for $y = \alpha+\beta J \in V(f) \cap \s_x$ and $a_2=(1\pm e_{123})h$ with $h \in \rr_2^*\subseteq C_{\rr_3}$ we have that $y':=-h^{-1}yh$ belongs to $\s_x$ and to $V(f^c)$. Indeed, $f(y) = a_1+Ja_2=0$ implies that $f^c(y') = a_1^c - h^{-1}Jh a_2^c = a_2^c J - h^{-1}J h h^c (1\pm e_{123}) = a_2^cJ - h^c(1\pm e_{123})J = a_2^cJ - a_2^cJ=0$. Moreover, in such a case $\s_x\subseteq V(N(f))$ as a consequence of the compatibility of $\rr_3$ and of Theorem~\ref{zeros}.

By what we already proved, if $V(f) \cap \s_x=\emptyset$ then automatically $V(f^c) \cap \s_x=\emptyset$. If $f'_s(x) \in C_A$ then by Theorem~\ref{zeros} $V(N(f)) \cap \s_x=\emptyset$, too. We conclude the proof by checking that the same holds when $f'_s(x)$ is a zero divisor. We first observe that $V(N(f)) \cap \s_x\neq\emptyset$ if, and only if, $\s_x \subseteq V(N(f))$: this is a consequence of formula $N(f)(\alpha+\beta I) = n(a_1)-n(a_2)+I t(a_1a_2^c)$, since in $\rr_3$ the functions $n,t$ take values in the center $\R+e_{123}\R$ of the algebra. Now, supposing $\s_x \subseteq V(N(f))$, that is, $n(a_1)=n(a_2)$ and $t(a_1a_2^c)=0$ with $a_2=(1\pm e_{123})a, a\in\R^*_2$ we will find a contradiction. The fact that $n(a_1)=n(a_2) =2(1\pm e_{123})n(a)$ implies that $a_1=(1\pm e_{123})b$, with $b\in\R_2^*$ having $n(b)=n(a)$. Setting $K:=-ba^{-1}$ we have that $n(K)=1$ and $0=t(a_1a_2^c)=(2\pm 2e_{123})t(ba^c)$. Then $K\in\s_{\R_3}$ and $f(\alpha+\beta K)=(1\pm e_{123})(b+Ka)=0$, which contradicts the hypothesis $V(f) \cap \s_x=\emptyset$.
 \end{proof}


\subsection{The slice regular case}\label{sec:zerosregular}

We are now ready to generalize a result known as the identity principle for slice regular quaternionic functions,~\cite[Theorem 1.12]{librospringer}, as well as its extensions to a larger class of quaternionic domains,~\cite[Theorem 3.6]{altavilla}; to slice regular octonionic functions,~\cite[Theorem 3.1]{rocky}; to slice monogenic functions,~\cite[Theorem 3.9]{israel}; and to admissible slice regular functions,~\cite[Theorem~20]{perotti}. For $J\in\s_A$, we will use the notations $\cc_J^+=\{\alpha+\beta J\in\cc_J\,:\, \beta>0\}$ and  $\OO_J^+:=\OO\cap\cc_J^+$. We also recall that we defined $|x|:=\sqrt{n(x)}$ for all $x \in Q_A$.

\begin{theorem}\label{discretezeros}
Let $\OO$ be open and let $f\in\mc{SR}(\OO)$. If $\OO$ is a slice domain then one of the following mutually exclusive properties holds:
\begin{enumerate}
\item[1.]
For each $J\in\s_A$, the intersection $V(f)\cap \cc_J$ is closed and discrete in $\OO_J$.
\item[2.]
$f\equiv0$.
\end{enumerate}
If $\OO$ is a product domain then, in addition to cases 1 and 2, there is one further possibility:
\begin{enumerate}
\item[3.] 
There exists $J_0\in\s_A$ such that $J=J_0$ satisfies
\begin{equation}\label{fin}
f_{|_{\OO_{J}^+}} \equiv 0 \mathrm{\ and\ } V(f)\cap\cc_{-J}^+ \mathrm{\ is\ closed\ and\ discrete\ in\ }\OO_{-J}^+.
\end{equation}
The set of $J \in \s_A$ for which~\eqref{fin} holds is determined by the equation $(J-J_0)f'_s \equiv 0$; in particular it reduces to $\{J_0\}$ if $f'_s$ takes at least one value that is neither $0$ nor a right zero divisor.
For all other $K\in \s_A$, the intersection $V(f)\cap\cc_K^+$ is closed and discrete in $\OO_K^+$.
\end{enumerate}
\end{theorem}

\begin{proof}
Assume that property {\it 1} does not hold and let $J_0\in\s_A$ be such that $V(f)\cap \OO_{J_0}$ has an accumulation point $z_0\in\OO_{J_0}$. Let  $f_{|_{\OO_{J_0}}}=\sum_{\ell=0}^h f_\ell J_{\ell}$ be the decomposition given by Lemma~\ref{splitting}. Then each $f_\ell$ has zeros that accumulate to $z_0$. Since $f_\ell$ is holomorphic, it must vanish identically on the connected component of $\OO_{J_0}$ containing $z_0$.
\begin{itemize}
\item If $\OO$ is a slice domain then $\OO_{J_0}$ is connected, whence $f_{|_{\OO_{J_0}}} \equiv 0$.  By formula~\eqref{rep2}, we conclude that $f \equiv 0$.
\item If $\OO$ is a product domain then $\OO_{J_0}$ has two connected components, namely $\OO_{J_0}^+$ and $\OO_{-{J_0}}^+$. If this is the case, then either of the following holds:
\begin{itemize}
\item $f \equiv 0$ in both $\OO_{J_0}^+,\OO_{-{J_0}}^+$, whence $f \equiv 0$ by formula~\eqref{rep2}; 
\item~\eqref{fin} holds for either $J=J_0$ or $J = -J_0$. Without loss of generality, we may suppose that it does for $J = J_0$. The fact that $f(x) = \vs f(x)+ \im(x) f'_s(x)$ vanishes for $x \in \OO_{J_0}^+$ implies that, for all $x \in \Omega$,
\[f(x) = \left(\im(x)-|\im(x)|\,J_0\right) f'_s(x);\]
while the discreteness of $V(f)\cap\cc_{-J_0}^+$ implies the discreteness of $V(f'_s)\cap\cc_{-J_0}^+$. Therefore, the intersection $V(f)\cap\cc_J^+$ is closed and discrete in $\OO_J^+$ for all $J\in \s_A\setminus\{J_0\}$, with the possible exception of those such that $(J-J_0) f'_s(x) \equiv 0$. In order for the last equation to admit a solution $J$, the nonzero values of $f'_s$ must all be right zero divisors in $A$. \qedhere
\end{itemize}
\end{itemize}
\end{proof}

Here are a few examples of case $3.$ in Proposition~\ref{discretezeros}.

\begin{example}[\cite{perotti}, Remark~12]\label{1fin}
Fix $J_0\in\s_A$. Let $f(x)=1+\frac{\im(x)}{|\im(x)|}J_0$ for each $x\in Q_A\setminus\rr$. Then $f$ is slice regular on $Q_A\setminus\rr$ and $V(f)=\cc_{J_0}^+$. Moreover $N(f)=N(f^c)\equiv0$.
\end{example}

Surprisingly, we can construct an example where~\eqref{fin} holds for more than one $J \in \s_A$.

\begin{example}\label{2fins}
Let $A = \rr_3$ and let $f(x)=\left(e_1 - \frac{\im(x)}{|\im(x)|}\right)(1-e_{123})$. Then $f$ is slice regular on $Q_A\setminus\rr$ and $V(f)=\cc_{e_1}^+ \cup \cc_{e_{23}}^+$. An example of $g \in \mc{SR}(Q_A\setminus\rr)$ with the same zero set $\cc_{e_1}^+ \cup \cc_{e_{23}}^+$, but which is not constant along the half-slices $\cc_J^+$, can be constructed following~\cite{altavilla} and letting $g(x)=x\cdot f(x) = x f(x)$.
\end{example}

We are now in a position to characterize the nonsingularity of $\mc{SR}(\OO)$.

\begin{proposition}\label{SRnonsingular}
Assume that $\OO$ is open. The $^*$-algebra $\mc{SR}(\OO)$ is nonsingular if and only if $A$ is nonsingular and $\OO$ is a union of slice domains.
\end{proposition}

\begin{proof}
If $A$ is singular then the algebra $\mc{SR}(\OO)$ is singular because of Example~\ref{ex:singular}.
If $\OO$ includes a product domain $\tilde \OO$ then we may define $f$ as in Example~\ref{1fin} on $\tilde \OO$ and set $f\equiv 0$ on $\OO \setminus \tilde \OO$. In this case we clearly have $f \not \equiv 0$ and $N(f) \equiv 0$.

Now let us suppose that $A$ is nonsingular and that $\OO$ consists only of slice domains. We can prove that $\mc{SR}(\OO)$ is nonsingular by the same argument used in~\cite[Theorem~20]{perotti}: if $N(f)\equiv 0$ then for all $x\in \OO \cap \rr$ we have $0 = N(f)(x) = n(f(x))$, whence $f(x)=0$ by the nonsingularity of $A$. Therefore, $f$ vanishes on the set $\OO \cap \rr$ and, by Theorem~\ref{discretezeros}, we conclude that $f \equiv 0$.
\end{proof}

Now that we have characterized nonsingularity for $\mc{SR}(\OO)$, the following consequence of Proposition~\ref{reciprocal} provides a useful tool. We point out that it is an exact generalization of the construction that is valid for slice regular quaternionic functions,~\cite[\S 5.1]{librospringer}.

\begin{proposition}
If $\OO$ is open and $\mc{SR}(\OO)$ is nonsingular then each element $f$ that is tame and not identically $0$ admits a reciprocal, namely $f^{-\punto} = N(f)^{-1} f^c$, on the open subset $\OO \setminus V(N(f))$ of $Q_A$, where $V(N(f))$ is a closed circular subset of $\OO$ with empty interior.
\end{proposition}

In the quaternionic and octonionic cases, it was actually possible to prove stronger results on the zeros of slice regular functions: see~\cite[\S 3.3]{librospringer} and~\cite[Theorem 1]{ghiloni}. These results have been recovered in~\cite{perotti} for admissible slice regular functions over $A$. We are now able to state them under different assumptions.

\begin{proposition}
Assume that $\OO$ is a slice domain or a product domain. Let $f\in\mc{SR}(\OO)$ be such that $V(f)\setminus\s_x$ has an accumulation point in $\s_x$. Under any of the following hypotheses, the normal function $N(f^c)$ vanishes identically:
\begin{enumerate}
\item
$A$ is a compatible $^*$-algebra.
\item
$V(N(f^c))$ is circular.
\item
$f$ is tame.
\item
$f_s'(x)\in C_A$ for each $x\in\OO\setminus\rr$.
\end{enumerate}
If, moreover, $\mc{SR}(\OO)$ is nonsingular then  $f\equiv0$.
\end{proposition}
\begin{proof}
Under any of the hypotheses listed in the statement, thanks to Corollaries~\ref{corollaryzeros} and~\ref{corollaryzeroscompatible}, the zero set $V(N(f^c))$ includes a circular set that accumulates to $\s_x$. By Theorem~\ref{discretezeros}, $N(f^c)\equiv0$.
\end{proof}

By the previous result and by Corollary~\ref{corollaryzeros}: 

\begin{corollary}\label{structure_zeros}
Assume that $\OO$ is a slice domain or a product domain.
Let $f \in \mc{SR}(\Omega)$ be such that $N(f^c)\not\equiv0$. If $f'_s(x) \in C_A$ for each $x \in \OO\setminus\rr$, then $V(f)$ is a union of isolated points or isolated spheres $\s_x$. 
\end{corollary}
 

\section{Zeros of slice products}\label{sec:zerosofproducts}

We now study some properties of the zeros of the slice product of two slice functions. The description was relatively simple in the case of slice regular functions over the quaternions,~\cite[Chapter 3]{librospringer}, and it became more articulate in the octonionic case,~\cite{ghiloni}, where the so-called ``camshaft effect'' was observed. We address here the general case of slice functions over $A$ whose spherical derivatives are either $0$ or invertible. We will then prove stronger properties under additional assumptions on $A$ or restricting to slice regular functions.
       
       \begin{theorem}\label{Camshaft}
           Let $f,g\in\mc{S}(\OO)$.  If $x\in\OO$ is real and $x\in V(f)\cup V(g)$, then $x\in V(f\cdot g)$. Now let $x\in\OO\setminus\rr$:
                  \begin{enumerate} 
                  \item 
	If $\s_x\subseteq V(f)$ or $\s_x\subseteq V(g)$, then $\s_x\subseteq V(f\cdot g)$.
            \item
            If $\s_x\cap V(f)$ includes a point $y$ then 
            \[(f\cdot g)'_s(x)=f'_s(x)\vs g(x)-(\im(y) f'_s(x)) g'_s(x).\]
            \begin{enumerate}
                \item If $f'_s(x)$ is invertible and $(f\cdot g)'_s(x)=0$ then $\s_x\subseteq V(f\cdot g)$.
                \item If $(f\cdot g)'_s(x)$ is invertible, then $\s_x\cap V(f\cdot g)\subseteq\{w\}$ with
                \[ w=\big((y f'_s(x)) \vs g(x)-(y\im(y) f'_s(x)) g'_s(x)\big) \big(  (f\cdot g)'_s(x) \big)^{-1}.\]
            \end{enumerate}
            \item
            If $\s_x\cap V(g)$ includes a point $z$,  then 
            \[(f\cdot g)'_s(x)=\vs  f(x) g'_s(x)- f'_s(x)(\im(z)  g'_s(x)).\]
            \begin{enumerate}
                \item If $g'_s(x)$ is invertible and $(f\cdot g)'_s(x)=0$ then $\s_x\subseteq V(f\cdot g)$.
                \item If $(f\cdot g)'_s(x)$ is invertible, then  $\s_x\cap V(f\cdot g)\subseteq\{w\}$ with
                \[ w=\big(\vs f(x)(z g'_s(x))- f'_s(x)(z\im(z) g'_s(x)) \big) \big( (f\cdot g)'_s(x) \big)^{-1}.\]
            \end{enumerate}            
            \item
           If, for some $y,z \in \s_x$, $y \in V(f)$ and $z\in V(g)$, then 
           \[(f\cdot g)'_s(x) =(y^c f'_s(x)) g'_s(x)-f'_s(x) (z g'_s(x)).\]
           \begin{enumerate}
               \item If $f'_s(x)$ or $g'_s(x)$ is invertible and $(f\cdot g)'_s(x)=0$ then $\s_x\subseteq V(f\cdot g)$.
               \item If $(f\cdot g)'_s(x)$ is invertible then $\s_x\cap V(f\cdot g)\subseteq\{w\}$, with 
               \[ w=\big( n(x) f'_s(x) g'_s(x) - (y f'_s(x)) (z g'_s(x))\big) \big( (f\cdot g)'_s(x) \big)^{-1}.\]
            \end{enumerate}
        \end{enumerate}
        \end{theorem}
       
       \begin{proof}  
           If $x\in\OO\cap\rr$, then $(f\cdot g)(x)=\vs  f(x) \vs g(x)=f(x)g(x)$. Otherwise, let $x\in\OO\setminus\rr$ and let $\s_x = \alpha + \beta \s$.           
           For all $I \in \s_A$, $f(\alpha+\beta I)=a_1+Ia_2$  where $a_1 = \vs  f(x), a_2 = \beta  f'_s(x)$, and $g(\alpha+\beta I)=b_1+Ib_2$  where $b_1 = \vs  g(x), b_2 = \beta  g'_s(x)$. Moreover,  $(f\cdot g)(\alpha+\beta I)=a_1b_1-a_2b_2 + I (a_1b_2+a_2b_1)$.
        \begin{enumerate}
        \item Assume that $f\equiv0$ on $\s_x$. Then $a_1=a_2=0$ and $f\cdot g\equiv0$ on $\s_x$. The case when $g\equiv0$ on $\s_x$ is analogous.   
        
        \item If $y=\alpha+\beta J \in V(f)$, then $a_1+Ja_2=0$ and
        \[(f\cdot g)(\alpha+\beta I)= -(Ja_2)b_1-a_2b_2+ I \left((-Ja_2)b_2+a_2b_1\right).\]
        Therefore $(f\cdot g)'_s(x)=-(\im(y) f'_s(x)) g'_s(x) +f'_s(x)\vs g(x)$. This quantity vanishes if and only if $a_2b_1=(Ja_2)b_2$, which is in turn equivalent (if $a_2$ is invertible) to the fact that
         \begin{align*}
         (f\cdot g)(\alpha+\beta I)&= -(Ja_2)(a_2^{-1}((Ja_2)b_2))-a_2b_2\\
         &= -((Ja_2)a_2^{-1}(Ja_2))b_2-a_2b_2\\
         &=-(J(Ja_2))b_2-a_2b_2=0
          \end{align*}
        for all $I \in \s$, owing to the first Moufang identity~\eqref{moufang1} and to property {\it 1} in Lemma~\ref{nonassociative}.
        
        Now let $(f\cdot g)'_s(x)$ be invertible. If $\s_x\cap V(f\cdot g)$ is not empty, then its unique element $w$ is given by  $w=\alpha-\vs {(f\cdot g)}(x) ((f\cdot g)'_s(x))^{-1}$.  We conclude by observing that:
        \begin{align*}
        \alpha(f\cdot g)'_s(x)-\vs {(f\cdot g)}(x) &= \alpha\beta^{-1}\left((-Ja_2)b_2+a_2b_1\right)+(Ja_2)b_1+a_2b_2\\
        &=  \beta^{-1}((-\alpha J+\beta)a_2) b_2+\beta^{-1}((\alpha+\beta J)a_2)b_1\\
        &=  -\beta^{-1}((\alpha +\beta J) Ja_2) b_2+\beta^{-1}((\alpha+\beta J)a_2)b_1\\
        &= -(y \im(y) f'_s(x))g'_s(x) + (y f'_s(x)) \vs g(x).
        \end{align*}
                
        \item If $z=\alpha+\beta K\in V(g)$, then $b_1+Kb_2=0$ and
        \[(f\cdot g)(\alpha+\beta I)= -a_1(Kb_2)-a_2b_2+ I \left(a_1b_2-a_2(Kb_2)\right).\]
        The slice derivative $(f\cdot g)'_s(x)=\vs  f(x) g'_s(x)- f'_s(x)(\im(z)  g'_s(x))$ vanishes if and only if  $a_1b_2=a_2(Kb_2)$, which is in turn equivalent (if $b_2$ is invertible) to
         \begin{align*}
         (f\cdot g)(\alpha+\beta I)&= -((a_2(Kb_2))b_2^{-1})(Kb_2)-a_2b_2\\
         &= -a_2((Kb_2)b_2^{-1}(Kb_2))-a_2b_2\\
         &=-a_2(K(Kb_2))-a_2b_2=0
          \end{align*}
         for all $I \in \s$, by the second Moufang identity~\eqref{moufang2} and by property {\it 1} of Lemma~\ref{nonassociative}.
        Now suppose $(f\cdot g)'_s(x)$ is invertible. If $\s_x\cap V(f\cdot g)$ is not empty, then its unique element $w$ is given by  $w=\alpha-\vs {(f\cdot g)}(x) ((f\cdot g)'_s(x))^{-1}$. By direct computation,
        $\alpha(f\cdot g)'_s(x)-\vs {(f\cdot g)}(x)=\vs f(x)(z g'_s(x))- f'_s(x)(z\im(z) g'_s(x))$.

        \item If $y=\alpha+\beta J \in V(f), z=\alpha+\beta K \in V(g)$ then $a_1+Ja_2=b_1+Kb_2=0$. It follows that $\vs f(x) = -\im(y)f'_s(x)$ and $\vs g(x) = -\im(z)f'_s(x)$ and the thesis can easily be deduced from cases {\it 2} and {\it 3}.
        \end{enumerate}
        \end{proof}
        
Here is an instance of case {\it 3.(b)} of Theorem~\ref{Camshaft}, where the intersection is actually empty. Examples of other cases will be given in Subsection~\ref{ss:compatible}.
   
\begin{example}
Let $A=\s\oo$ with the $^*$-involution introduced in Example~\ref{soincompatible} and let $f(x)=x-2l$, $g(x)=x-i$. Then $f,g \in \mc{SR}(Q_A)$ have slice derivatives $f'_s\equiv 1 \equiv g'_s$ and clearly $\s_A \cap V(f) = \emptyset, \s_A \cap V(g)=\{i\}$. On the other hand, if we look at their slice product $f\cdot g(x) = x^2-x(i+2l) + 2li = x^2+1-\left(x+\frac53 i+ \frac 43 l\right)(i+2l)$ on the sphere $\s_A$, we note that the slice derivative $-(i+2l)$ is invertible and that the value $-\left(\frac53 i+ \frac 43 l\right)$ does not belong to $\s_A$; indeed, $t\left(\frac53 i+ \frac 43 l\right) = \frac83 l \neq 0$. Therefore, $\s_A \cap V(f\cdot g) = \emptyset$. 
 \end{example}

It is natural to also consider the case when $\s_x\cap V(f)=\s_x\cap V(g)=\emptyset$. In general, we cannot conclude that $\s_x\cap V(f\cdot g)=\emptyset$: for instance, this is not the case if $f\equiv a$ and $g \equiv b$ where $a,b$ are zero divisors such that $ab=0$. However, when $f$ and $g$ are tame we can prove the next result.

\begin{proposition}\label{Vfg}
Let $f,g\in\mc{S}(\OO)$ be tame. Then it holds
\[
V(f\cdot g)\subseteq V(N(f))\cup V(N(g)).
\]
If $A$ is nonsingular and, for each $x\in\OO\setminus\rr$, $f'_s(x)$ and $g'_s(x)$ belong to $C_A$, then 
\[
V(f\cdot g)\subseteq\bigcup_{x\in V(f)\cup V(g)}\s_x.
\]
\end{proposition}
       
\begin{proof}
By Proposition~\ref{reciprocal}, $f\cdot g$ is tame and $N(f \cdot g)=N(f)N(g)$ on $\tilde\OO:=\OO\setminus(V(N(f))\cup V(N(g)))$. If there existed $x\in\tilde\OO\cap V(f\cdot g)$, then $x^c\in \tilde\OO\cap V(N((f\cdot g)^c))=\tilde\OO\cap V(N(f\cdot g))$. Thus $N(f\cdot g)=N(f) N(g)$ would vanish on the sphere $\s_x\subseteq\tilde\OO$. But then $x$ would belong to $V(N(f))\cup V(N(g))$, a contradiction. Therefore $V(f\cdot g)\subseteq V(N(f))\cup V(N(g))$. The last part is a consequence of Corollary~\ref{corollaryzeros}, case {\it 2}.
\end{proof}

We have already observed that the nonsingularity hypothesis is essential for the second statement of Proposition~\ref{Vfg}, in view of examples such as~\ref{ex:singular} and~\ref{ex:singular2}. We now show that the tameness hypothesis cannot either be removed from Proposition~\ref{Vfg}.

\begin{example}
Let $A=\s\oo$ with the $^*$-involution introduced in Example~\ref{soincompatible} and let $f(x)=x-li$, $g(x)=x+li$. Then $f,g \in \mc{SR}(Q_A)$ and their slice product $(f \cdot g)(x) = x^2-1$ vanishes at $1$. On the other hand, by direct computation neither $N(f)(x) = x^2-2x(li)+1$ nor $N(g)(x) = x^2+2x(li)+1$ vanish at $1$.
\end{example}

  
\subsection{The associative or compatible case}\label{ss:compatible}

In the compatible setting, we can strengthen Theorem~\ref{Camshaft} under additional hypotheses on the slice derivatives.
  
\begin{theorem}\label{Camshaftcompatible}
Assume that $A$ is compatible.
    Let $f,g\in\mc{S}(\OO)$.  If $x\in\OO$ is real and $x\in V(f)\cup V(g)$, then $x\in V(f\cdot g)$. More generally,
                  \begin{enumerate} 
                  \item 
	If $\s_x\subseteq V(f)$ or $\s_x\subseteq V(g)$, then $\s_x\subseteq V(f\cdot g)$.
	\end{enumerate}
If $x\in\OO\setminus\rr$ and $f'_s(x), g'_s(x)$ and $(f\cdot g)'_s(x)$ belong to $C_A$, then:
       \begin{enumerate}
                   \item[2.]
            If $\s_x\cap V(f)$ is a singleton $\{y\}$ and $\s_x\cap V(g)=\emptyset$, then
	$\s_x\cap V(f\cdot g)\subseteq\{w\}$, with
                \[ w=\big((y f'_s(x)) \vs g(x)-(y\im(y) f'_s(x)) g'_s(x)\big) \big( f'_s(x)\vs g(x)-(\im(y) f'_s(x)) g'_s(x) \big)^{-1}.\]
            \item[3.] 
            If $\s_x\cap V(f)=\emptyset$ and $\s_x\cap V(g)$ is a singleton $\{z\}$,  then $\s_x\cap V(f\cdot g)\subseteq\{w\}$, with
                \[ w=\big(\vs f(x)(z g'_s(x))- f'_s(x)(z\im(z) g'_s(x)) \big) \big( \vs  f(x) g'_s(x)- f'_s(x)(\im(z)  g'_s(x)) \big)^{-1}.\]   
           \item[4.]
           If $\s_x\cap V(f)=\{y\}$ and $\s_x\cap V(g)=\{z\}$ for some $y,z \in \s_x$, then one of the following holds:
           \begin{enumerate}
               \item  $\s_x\subseteq V(f\cdot g)$; or
               \item  $\s_x\cap V(f\cdot g)\subseteq\{w\}$, with 
               \[ w=\big( n(x) f'_s(x) g'_s(x) - (y f'_s(x)) (z g'_s(x))\big) \big( (f\cdot g)'_s(x) \big)^{-1};\]
            \end{enumerate}
            depending on whether or not $(f\cdot g)'_s(x) =(y^c f'_s(x)) g'_s(x)-f'_s(x) (z g'_s(x))$ vanishes.
     
        \end{enumerate}
        \end{theorem}
 
  Before proving the theorem, we need one more algebraic property:
   
\begin{lemma}\label{traceless}
If $A$ is a compatible $^*$-algebra, then the trace function vanishes on any associator.
      \end{lemma}      
\begin{proof}
Given any $x,y,z$ in $A$, 
\[8(x,y,z)=(2x-t(x),2y-t(y),2z-t(z))=(x-x^c,y-y^c,z-z^c).\]
Hence $(x,y,z)^c=(x^c,y^c,z^c)=\frac{1}{8}(x^c-x,y^c-y,z^c-z)=-(x,y,z)$.
 \end{proof}

\begin{proof}[Proof of Theorem~\ref{Camshaftcompatible}]
We only have to rule out cases {\it 2.(a)} and {\it 3.(a)} of Theorem~\ref{Camshaft}. We resume the formulas obtained in its proof, based on the expressions:
 $f(\alpha+\beta I)=a_1+Ia_2$, $g(\alpha+\beta I)=b_1+Ib_2$  where  $a_1 = \vs  f(x), a_2 = \beta  f'_s(x)$, $b_1 = \vs  g(x), b_2 = \beta  g'_s(x)$. 

Case {\it 2.(a)} takes place if, and only if, $a_2$ is invertible and $(Ja_2)b_2=a_2b_1$, two facts which in turn imply that $a_2^{-1}( (Ja_2)b_2 ) = b_1$. According to our present hypotheses, $a_2 \in C_A^*$ and $b_2 \in C_A$. Either $b_2=0=b_1$, contradicting the hypothesis $V(g) \cap \s_x = \emptyset$; or $b_2 \in C_A^*$ and we can set 
\begin{equation*}
U:=-b_1b_2^{-1}=- (a_2^{-1}( (Ja_2)b_2 )) b_2^{-1}.
\end{equation*}
In this case, owing to Proposition~\ref{compatiblecasenorm}, $n(U) = n(J) = 1$. Moreover, 
\begin{equation*}
t(U) = -t\left( a_2^{-1} (((Ja_2)b_2 ) b_2^{-1}) \right) = -t \left( a_2^{-1}Ja_2 \right) =0,
\end{equation*}
where the first equality follows from Lemma~\ref{traceless}, the second one follows by case {\it 1} of Lemma~\ref{nonassociative} and the third is part of Remark~\ref{preservedsphere}. We conclude that $U \in \s$ and $g(\alpha+\beta U)=b_1+Ub_2=0$, a contradiction.
        
Case {\it 3.(a)} takes place if, and only if, $b_2$ is invertible and $a_1b_2=a_2(Kb_2)$. According to our present hypotheses, $a_2 \in C_A$ and $b_2 \in C_A^*$. Case $a_2=0$ is excluded, for it would imply $a_1=0$ and contradict the hypothesis $V(f) \cap \s_x = \emptyset$. Thus $a_2,b_2\in C_A^*$. Proceeding as above, we see that $W:=-a_2^{-1}a_1=-a_2^{-1}((a_2(Kb_2))b_2^{-1})$ belongs to $\s_A$. Therefore, $I:=-a_1a_2^{-1}=a_2Wa_2^{-1} \in \s_A$ by Remark~\ref{preservedsphere} and $f(\alpha +I \beta) = a_1+I a_2 = 0$, a contradiction.
 \end{proof}

An illustration of cases {\it 2} and {\it 4.(a)} of Theorem~\ref{Camshaftcompatible} follows. Other cases will be covered in Examples~\ref{ex:clifford},~\ref{ex:clifford1} and~\ref{ex:clifford2}.
 
\begin{example}
Let $A=\s\oo$ and let $f(x)=xi-j$, $g(x)=l$. Then $f$ and $g$ are slice regular on $Q_A$ and $(f\cdot g)(x)=x(il)-jl=(x+k)(il)$. By direct inspection, $V(f)=\{k\}$, $V(g)=\emptyset$ while $V(f\cdot g)=\{-k\}$. 
\end{example}
   
\begin{example}
Let $A=\s\oo$ and let $f(x)=xl-il$, $g(x)=x-i$. Then $f$ and $g$ are slice regular on $Q_A$ and $(f\cdot g)(x)=(x^2+1)l$. By direct inspection, $V(f)=V(g)=\{i\}$ while $V(f\cdot g)=\s_A$.
\end{example}

In the compatible case, Proposition~\ref{Vfg} strengthens as well:

\begin{proposition}\label{vfgcompatible}
Suppose $A$ is compatible and nonsingular. Let $f,g\in\mc{S}(\OO)$ be tame. If, for each $x \in \OO \setminus \rr$, $f'_s(x), g'_s(x)$ and $(f\cdot g)'_s(x)$ belong to $C_A$, 
     then
    \[\bigcup_{x\in V(f\cdot g)}\s_x =\bigcup_{x\in V(f)\cup V(g)}\s_x.\]
\end{proposition}

\begin{proof}
By Remark~\ref{compatibleN}, we know that $N(f\cdot g) = N(f) N(g)$ and that $N(f\cdot g)$ vanishes if, and only if, $N(f)$ or $N(g)$ does. Therefore $V(N(f\cdot g)) = V(N(f))\cup V(N(g))$. 
    The  statement is then a direct consequence of Corollary~\ref{corollaryzeroscompatible}.
\end{proof}

Here is an example (within an associative nonsingular algebra) where Proposition~\ref{vfgcompatible} does not hold because $f,g$ do not fulfill all of its hypotheses.

\begin{example}
Let $A=\rr_n$ (with $n\ge4$) and set $f\equiv1-2e_{123}$, $g(x)=x-e_4$. Then $f$ and $g$ are slice regular on $Q_A$, they have slice derivatives $f'_s\equiv0 \in C_A, g'_s\equiv1 \in C_A$ and $V(f) = \emptyset,V(g)=\{e_4\}$. The product $(f\cdot g)(x)=x(1-2e_{123})-(1-2e_{123})e_4$ has no zeros in $Q_A$, since $1-2e_{123}$ is invertible and
\begin{align*}
&(1-2e_{123})e_4(1-2e_{123})^{-1} = -\frac 1 3 (1-2e_{123})e_4 (1+2e_{123}) \\
&= -\frac 1 3 (1-2e_{123})(1-2e_{123})e_4 = -\frac 1 3 (5-4e_{123})e_4 = -\frac 5 3 e_4 + \frac 4 3 e_{1234}
\end{align*}
has trace $t\left((1-2e_{123})e_4(1-2e_{123})^{-1}\right) = \frac 8 3 e_{1234}$.
\end{example}

Thanks to Proposition~\ref{vfgcompatible}, if $A$ is nonsingular and $f$ and $g$ are tame, then the inclusions appearing in cases {\it 2}, {\it 3} and {\it 4.(b)} of Theorem~\ref{Camshaftcompatible} become equalities. In the associative case, the same holds without assuming tameness, owing to the next proposition.

\begin{proposition}\label{pro_associative}
Let $A$ be associative. If $f'_s, g'_s$ and $(f\cdot g)'_s$ take values in $C_A$, then 
\[\bigcup_{x\in V(f\cdot g)}\s_x \supseteq\bigcup_{x\in V(f)\cup V(g)}\s_x.\] 
\end{proposition}
    
\begin{proof}
If $x\in V(f)$ then $x\in V(f\cdot g)$ by Corollary~\ref{formula_associative_case}. If $x\in V(g)$, from Theorem~\ref{zeros} we get that $\s_x\subseteq V(g^c)$ or $\s_x\cap V(g^c)=\{x'\}$ for some $x'$. Therefore $\s_x\subseteq V(g^c\cdot f^c)$ or $x'\in V(g^c\cdot f^c)$, respectively. In any case, $\s_x \cap V((f\cdot g)^c) \neq \emptyset$, where $(f\cdot g)^c=g^c\cdot f^c$. By formula~\eqref{decomposedconjugate}, the spherical derivative of $(f\cdot g)^c$ at $x$ belongs to $C_A$. In view of Theorem~\ref{zeros}, we obtain that $\s_x\cap V(f\cdot g)\neq \emptyset$.         
\end{proof}

In the associative setting, Theorem~\ref{Camshaftcompatible} can be restated in the following simpler form.

\begin{theorem}\label{Camshaftassociative}
Let $A$ be associative. Let $f,g \in \mc{S}(\OO)$. If $x \in \OO$ is real and $x \in V(f)\cup V(g)$, then $x \in V(f\cdot g)$. More generally,
\begin{enumerate} 
 \item If $\s_x \subseteq V(f)$ or $\s_x \subseteq V(g)$, then $\s_x \subseteq V(f\cdot g)$.
\end{enumerate}
If $x \in \OO \setminus \rr$, and $f'_s(x), g'_s(x)$ and $(f\cdot g)'_s(x)$ belong to $C_A$, then:
\begin{enumerate}
\item[2.] If $\s_x \cap V(f)=\{y\}$ and $\s_x \cap V(g)=\emptyset$, then $\s_x \cap V(f\cdot g)=\{y\}$.
\item[3.] If $\s_x \cap V(f)=\emptyset$ and $\s_x \cap V(g)=\{z\}$, then $f^c(z) \in C_A^*$ and $\s_x\cap V(f\cdot g)=\{f^c(z)^{-1}z f^c(z)\}$.
\item[4.] If $\s_x \cap V(f)=\{y\}$ and $\s_x \cap V(g)=\{z\}$, then one of the following holds:
\begin{enumerate}
\item $\s_x\subseteq V(f\cdot g)$; or
\item $\s_x\cap V(f\cdot g)=\{y\}$;
\end{enumerate}
depending on whether or not $y^c f'_s(x)=f'_s(x)z$.   
\end{enumerate}
\end{theorem}

\begin{proof}
Point {\it 1} comes verbatim from Theorem~\ref{Camshaftcompatible}.
Point {\it 2} derives from case {\it 2} of the same theorem, if we take into account that in the associative setting $y \in V(f)$ implies $y \in V(f\cdot g)$ (see Corollary~\ref{formula_associative_case}).

As for point {\it 3}, consider the homologous case of Theorem~\ref{Camshaftcompatible}. In the present associative setting, by Proposition~\ref{pro_associative}, the fact that $z \in \s_x \cap V(g)$ implies the existence of a $w \in \s_x\cap V(f\cdot g)$. By the formula given for $w$ in Theorem~\ref{Camshaftcompatible},
\begin{align*}
w&=\big(\vs f(x)(z g'_s(x))- f'_s(x)(z\im(z) g'_s(x)) \big) \big( \vs  f(x) g'_s(x)- f'_s(x)(\im(z)  g'_s(x)) \big)^{-1} \\
&=\big(\vs f(x)z- f'_s(x)z\im(z) \big) g'_s(x) g'_s(x)^{-1} \big( \vs  f(x)- f'_s(x)\im(z) \big)^{-1} \\
&=\big(\vs f(x)- f'_s(x)\im(z) \big) z \big( \vs  f(x)- f'_s(x)\im(z) \big)^{-1} \\
&=\big(\vs f(x)^c+\im(z)f'_s(x)^c \big)^{-1} z \big(\vs f(x)^c+\im(z)f'_s(x)^c \big) \\
&= f^c(z)^{-1}zf^c(z).
\end{align*}
In the last display, the first equality is justified by the associativity assumption; the second one uses the fact that $z\im(z) = \im(z)z$; the third equality takes into account that, for all $a \in Q_A$, $a^{-1} = n(a)^{-1}a^c$ with $n(a) \in \rr$; the last one uses formula~\eqref{decomposedconjugate}. Moreover, if one goes through the series of equalities taking into account that $g'_s(x),(f\cdot g)'_s(x) \in C_A^*$ it turns out that $f^c(z) \in C_A^*$ as well.

Point {\it 4} derives from the homologous case of Theorem~\ref{Camshaftcompatible} by the same reasoning used for point {\it 2}  and by noticing that
\[(y^c f'_s(x)) g'_s(x)-f'_s(x) (z g'_s(x)) = (y^c f'_s(x)-f'_s(x)z) g'_s(x)\]
vanishes if, and only if, $y^c f'_s(x)=f'_s(x)z$.
\end{proof}

Here are some classical examples of cases {\it 1}, {\it 3} and {\it 4.(b)} in the previous theorem.

\begin{example}\label{ex:clifford}
Let $A=\rr_n$ and let $f(x)=e_1$, $g(x)=x^2+1$. Then $f$ and $g$ are slice regular on $Q_A$ and $(f\cdot g)(x)=x^2e_1+e_1 = (x^2+1)e_1$. By direct inspection, $V(f)=\emptyset$, $V(g)=\s_A$ and $V(f\cdot g)=\s_A$.
\end{example}

\begin{example}\label{ex:clifford1}
Suppose $A=\rr_n$ with $n\geq2$ and set $f(x)=e_1$, $g(x)=x-e_2$. Then $f,g \in \mc{SR}(Q_A)$ and $(f\cdot g)(x)=xe_1-e_{12}=(x+e_2)e_1$. Clearly, $V(f)=\emptyset$, $V(g)=\{e_2\}$ while $V(f\cdot g)=\{-e_2\}$.
\end{example}

\begin{example}\label{ex:clifford2}
Let $A=\rr_n$ ($n\geq2$) and let $f(x)=x-e_1$, $g(x)=x-e_2$. Then $f$ and $g$ are slice regular on $Q_A$ and $(f\cdot g)(x)=x^2-x(e_1+e_2)+e_{12}=(x^2+1)-(x-e_1)(e_1+e_2)$. By direct inspection, $V(f)=\{e_1\}$, $V(g)=\{e_2\}$ while $V(f\cdot g)=\{e_1\}$.
\end{example}

We observe that Theorem~\ref{Camshaftassociative} echoes, under different hypotheses, the next remark (a direct consequence of Theorem~\ref{fg-associative}).
\begin{remark}
Suppose $A$ is associative. Let $f,g \in \mc{S}(\OO)$ with $f$ tame. If $y \in \OO$ is such that $f(y)=0$ then $(f \cdot g) (y) = 0$. For $x \in \OO$ such that $f(x)$ is invertible, $(f \cdot g) (x) = f(x) g(f(x)^{-1} x f(x))$ vanishes if, and only if, $g$ vanishes at $z=f(x)^{-1} x f(x)$; moreover, $f^c(z)$ is invertible and $x = f^c(z)^{-1} z f^c(z)$.
\end{remark}


\subsection{The slice regular case}\label{sec:fgregular}
  
For slice regular functions on a slice domain, Proposition~\ref{Vfg} strengthens even without the compatibility assumption on the algebra.

\begin{proposition}\label{unions}
Assume $A$ is nonsingular and $\OO$ is a slice domain. Let $f,g\in\mc{SR}(\OO)$ be tame. If, for each $x\in\OO\setminus\rr$, $f'_s(x), g'_s(x)$ and $(f\cdot g)'_s(x)$   belong to $C_A$, then
\[\bigcup_{x\in V(f\cdot g)}\s_x =\bigcup_{x\in V(f)\cup V(g)}\s_x.\]
\end{proposition}

\begin{proof}
By Proposition~\ref{SRnonsingular}, the $^*$-algebra $\mc{SR}(\OO)$ is nonsingular: if $N(f) \equiv 0$ or $N(g) \equiv 0$ then $f \equiv 0$ or $g\equiv 0$. In such a case, $f\cdot g \equiv 0$ so that the thesis is immediately verified.

We now consider the case when neither $N(f)$ nor $N(g)$ vanish identically. According to Proposition~\ref{Nfgregular}, $f\cdot g$ is tame in $\OO$ and $N(f\cdot g) = N(f) N(g)$ in the entire domain. As a consequence, $V(N(f\cdot g)) = V(N(f))\cup V(N(g))$, whence the thesis by Corollary~\ref{corollaryzeros}.
\end{proof}

We conclude showing that a slice regular function $f$ that is tame can be a zero divisor only if $N(f)\equiv0$. 
     
\begin{proposition}
Assume that $\OO$ is a slice domain or a product domain.
Let $f\in\mc{SR}(\OO)$ be tame, with $N(f)\not\equiv0$. Then, for every $g\in\mc{S}^0(\OO)$, 
$f\cdot g\equiv0$ or $g\cdot f\equiv0$ implies  $g\equiv0$. 
\end{proposition}
\begin{proof}
By Proposition~\ref{reciprocal} and Remark~\ref{dense}, $f$ admits a reciprocal $f^{-\punto}$ in a domain $\OO'$ that is dense in $\OO$. If $f\cdot g\equiv0$ then $g=(f^{-\punto}\cdot f)\cdot g=f^{-\punto}\cdot (f\cdot g)\equiv 0$ in $\OO'$ by Lemma~\ref{nonassociative}. By continuity, $g\equiv0$ in $\OO$. A similar argument can be used when $g\cdot f\equiv0$.
\end{proof}

The previous result immediately implies our final statement, which concerns the algebra of slice regular functions in a case in which it is not singular.
\begin{corollary}
If $A$ is nonsingular and $\OO$ is a slice domain, then each element $f\in\mc{SR}(\OO)$ that is tame cannot be a zero divisor in $\mc{S}^0(\OO)$.
\end{corollary}


\section*{Concluding remarks}

Algebra and analysis had already proven to be as beautifully interwoven in the theory of slice regular quaternionic functions as they are in the theory of holomorphy.
We believe that our present work shows that this is the case on all alternative $^*$-algebras. Moreover, we have tried to give an idea of how intricate the weave grows as we relax the hypotheses on the algebra under consideration.

Besides their intrinsic interest, our new results are meant to be the grounds for the study of singularities of slice regular functions over alternative $^*$-algebras. An article on this subject, in the spirit of the quaternionic work~\cite{singularities}, is already in preparation.



\bibliographystyle{abbrv}
\bibliography{Laurent}


\end{document}